\documentclass{tac}

\usepackage{amsmath}
\usepackage{amssymb}
\usepackage{enumitem}

\usepackage[all,cmtip]{xy}

\input diagxy

\usepackage[colorlinks=true]{hyperref}
\hypersetup{allcolors=[rgb]{0.1,0.1,0.4}}

\author{Cheng-Yong Du, Lili Shen and Xiaojuan Zhao}

\thanks{The first named author acknowledges the support of National Natural Science Foundation of China (11501393). The second named author acknowledges the support of National Natural Science Foundation of China (11701396) and the Fundamental Research Funds for the Central Universities (YJ201644). The third named author acknowledges the support of National Natural Science Foundation of China (11626195) and the Fundamental Research Funds for the Central Universities (2682017CX064).}

\address{School of Mathematics, Sichuan Normal University\\
 Chengdu 610068, China\\[5pt]
 School of Mathematics, Sichuan University\\
 Chengdu 610064, China\\[5pt]
 Department of Mathematics, Southwest Jiaotong University\\
 Chengdu 611756, China\\
}

\title{Spark complexes on good effective orbifold atlases categorically}

\copyrightyear{2018}


\eaddress{cyd9966@hotmail.com \CR shenlili@scu.edu.cn \CR xjzhao@swjtu.edu.cn}

\newtheorem{thm}{Theorem}

\newtheorem{lem}{Lemma}
\newtheorem{prop}{Proposition}

\newtheoremrm{rem}{Remark}
\newtheoremrm{defn}{Definition}
\newtheoremrm{exmp}{Example}

\mathrmdef{Hom}
\mathbfdef{Set}

\DeclareMathOperator{\sgn}{sgn}
\DeclareMathOperator{\Diff}{Diff}
\DeclareMathOperator{\Tot}{{\sf Tot}}

\def\toRr{\mathrel{\equiv\joinrel\Rrightarrow}}

\newcommand{\op}{{\rm op}}

\renewcommand{\phi}{\varphi}
\newcommand{\al}{\alpha}
\newcommand{\be}{\beta}

\newcommand{\de}{\delta}

\newcommand{\Ga}{\Gamma}

\newcommand{\lam}{\lambda}

\newcommand{\si}{\sigma}
\newcommand{\Om}{\Omega}
\newcommand{\om}{{a}}
\renewcommand{\eta}{{b}}

\newcommand{\CD}{\mathcal{D}}

\newcommand{\CS}{\mathcal{S}}

\newcommand{\CU}{\mathcal{U}}
\newcommand{\CV}{\mathcal{V}}
\newcommand{\CW}{\mathcal{W}}
\newcommand{\CX}{\mathcal{X}}

\newcommand{\sC}{{\sf C}}

\newcommand{\sE}{{\sf E}}
\newcommand{\sF}{{\sf F}}

\newcommand{\sI}{{\sf I}}

\newcommand{\bbC}{\mathbb{C}}

\newcommand{\bbR}{\mathbb{R}}

\newcommand{\bbZ}{\mathbb{Z}}

\newcommand{\BH}{{\bf H}}

\newcommand{\FI}{\mathfrak{I}}
\newcommand{\FJ}{\mathfrak{J}}

\newcommand{\tU}{\tilde{U}}
\newcommand{\tV}{\tilde{V}}
\newcommand{\tW}{\tilde{W}}

\newcommand{\hH}{\widehat{\BH}}
\newcommand{\hI}{\hat{I}}

\newcommand{\tf}{\tilde{f}}

\newcommand{\tg}{\tilde{g}}

\newcommand{\hi}{\hat{i}}
\newcommand{\hj}{\hat{j}}
\newcommand{\hk}{\hat{k}}
\newcommand{\tx}{\tilde{x}}

\newcommand{\utf}{\underline{\tf}}
\newcommand{\utg}{\underline{\tg}}

\newcommand{\ual}{\underline{\al}}
\newcommand{\ube}{\underline{\be}}
\newcommand{\ubeal}{\underline{\be\al}}
\newcommand{\ubecal}{\underline{\be\circ\al}}
\newcommand{\od}{\overline{d}}
\newcommand{\oC}{\overline{\sC}}

\newcommand{\oD}{\overline{D}}
\newcommand{\oE}{\overline{\sE}}
\newcommand{\oF}{\overline{\sF}}
\newcommand{\oI}{\overline{\sI}}
\newcommand{\oS}{\overline{\CS}}

\newcommand{\oH}{\overline{\hH}}

\newcommand{\oOm}{\overline{\Om}}
\newcommand{\FVU}{{\sf Fin}(V_{\CU})}
\newcommand{\FVV}{{\sf Fin}(V_{\CV})}

\newcommand{\MVU}{{\sf MFin}(V_{\CU})}
\newcommand{\MVV}{{\sf MFin}(V_{\CV})}

\newcommand{\MmnVU}{{\sf MFin}^{m+n}(V_{\CU})}
\newcommand{\MpVU}{{\sf MFin}^p(V_{\CU})}

\newcommand{\MppVU}{{\sf MFin}^{p+1}(V_{\CU})}

\newcommand{\MpppVU}{{\sf MFin}^{p+2}(V_{\CU})}

\newcommand{\MpmVU}{{\sf MFin}^{p-1}(V_{\CU})}

\newcommand{\tidU}{\widetilde{1_{\CU}}}
\newcommand{\dsum}{\displaystyle\sum}
\newcommand{\mVU}{{\sf M}(V_{\CU})}
\newcommand{\mpVU}{{\sf M}^p(V_{\CU})}
\newcommand{\ophi}{\overline{\phi}}

\newcommand{\Cat}{{\bf Cat}}

\newcommand{\PreOrb}{{\bf Pre\text{-}Orb}}
\newcommand{\GPreOrb}{{\bf GOrbAtl}}
\newcommand{\SpCx}{{\bf SpCx}}

\newcommand{\Char}{{\sf Char}}

\newcommand{\Ab}{{\bf Ab}}
\newcommand{\GAb}{\Ab^{\bbZ}}
\newcommand{\GCRng}{{\bf GCRng}}

\numberwithin{equation}{section}

\allowdisplaybreaks

\begin{document}

\maketitle
\begin{abstract}
Good atlases are defined for effective orbifolds, and a spark complex is constructed on each good atlas. It is proved that this process is 2-functorial with compatible systems playing as morphisms between good atlases, and that the spark character 2-functor factors through this 2-functor.
\end{abstract}

\section{Introduction}

Spark complexes, which unify several approaches to differential characters as considered in \cite{Gillet1989,Harvey1977,Harvey2003}, were introduced by Harvey and Lawson \cite{Harvey2006} as an algebraic description of secondary geometric invariants of smooth manifolds initiated by Cheeger and Simons \cite{Cheeger1985}. Explicitly, on a smooth manifold $X$ several spark complexes can be constructed, e.g., de Rham--Federer spark complexes and Cheeger--Simons spark complexes; the associated graded abelian groups of spark characters of these spark complexes, now known as that of \emph{Harvey--Lawson characters} and denoted by $\hH^*(X)$, are isomorphic to each other. Moreover, $\hH^*(X)$ is a commutative graded ring that is isomorphic to the ring $H^*_{CS}(X)$ of Cheeger--Simons differential characters on $X$ \cite{Cheeger1985} and the smooth Deligne--Beilinson cohomology $\bigoplus_q H^q_{\CD}(X,\bbZ(q)^{\infty})$ of $X$ \cite{Hao2009}.

The aim of this paper is to investigate spark complexes on \emph{orbifolds}, a generalization of manifolds, from the viewpoint of category theory. More precisely, our purpose is to find an approach to construct spark complexes on orbifolds, so that the functoriality of this process can be established.

Orbifolds have been studied mainly from two perspectives, one of which is through the language of \emph{orbifold charts} and \emph{orbifold atlases} as in Satake's first paper \cite{Satake1956} that introduced orbifolds under the name ``$V$-manifolds'', and the other is through \emph{orbifold groupoids} (i.e., proper {\'e}tale Lie groupoids) as first discovered by Moerdijk and Pronk \cite{Moerdijk1997}; we refer to \cite{Adem2007,Moerdijk2002} for an overview of the theory of orbifolds. During recent years, secondary geometric invariants of orbifolds have been investigated through proper {\'e}tale Lie groupoids by several authors; see, e.g., \cite{Felisatti2012,Lupercio2001,Lupercio2006}. As for the approach of orbifold atlases, Du and Zhao constructed in \cite{Du2016} a spark complex on an effective orbifold equipped with a \emph{good atlas} generated from a \emph{good triangulation} (see \cite[Subsection 1.2]{Moerdijk1999}), which is the motivation of this paper and will be referred to in the appendix (see Eq. \eqref{SU-def}).

Although the construction of the spark complex \eqref{SU-def} is feasible for any effective orbifold, it is far from establishing the functoriality. The reasons are: first, it is difficult to find appropriate morphisms between orbifolds; second, the spark complex \eqref{SU-def} does not interact well with embeddings in the corresponding good atlas of the orbifold. Therefore, in order to achieve our goal, inspired by Tommasini's works \cite{Tommasini2012,Tommasini2016} we consider \emph{orbifold atlases} instead of orbifolds (i.e., equivalence classes of orbifold atlases) as objects in the source category, so that \emph{compatible systems} can be employed as their morphisms.

More specifically, we axiomatize good atlases generated from Moerdijk and Pronk's good triangulation (Definition \ref{good-atlas}) and consider the 2-category
$$\GPreOrb$$
with good effective orbifold atlases as objects, compatible systems as 1-cells and their natural transformations as 2-cells (Proposition \ref{GPreOrb-Cat}), upon which spark complexes can always be constructed. Then, for every good atlas $\CU$ we construct a new spark complex $\oS_{\CU}$ (Theorem \ref{oSU-def}), and it is proved that this process gives rise to a 2-functor
$$\oS:\GPreOrb^{\op}\to\SpCx$$
into the 2-category of spark complexes, spark homomorphisms and spark homotopy classes (Theorem \ref{oS-2-functor}), whose proof is the most challenging one in this paper.

Furthermore, the associated group $\oH^*(\CU)$ of spark characters of the spark complex $\oS_{\CU}$ is also 2-functorial as
$$\oH:\GPreOrb^{\op}\to\GAb,$$
where $\GAb$ is the category of graded abelian groups and their homomorphisms (considered as a 2-category with trivial 2-cells), and the \emph{spark character 2-functor} $\oH$  factors through $\oS$ (Theorem \ref{oH-2-functor}).

This paper is a first step towards the functoriality of spark complexes on orbifolds, and many things remain to be discovered. We leave here two questions that could be taken into consideration in future works:
\begin{enumerate}[label=(\arabic*)]
\item As atlases for ineffective orbifolds have been postulated by Pronk, Scull and Tommasini \cite{Pronk2016}, is it possible to define good atlases and establish the 2-functor $\oS:\GPreOrb^{\op}\to\SpCx$ for ineffective orbifolds?
\item In the context of manifolds, Simons and Sullivan proved that there is only one spark character functor making the \emph{Character Diagram} commutative \cite{Simons2008}. Is it possible to deduce similar uniqueness for our spark character 2-functor $\oH:\GPreOrb^{\op}\to\GAb$?
\end{enumerate}

We thank the anonymous referee for several helpful remarks.

\section{The 2-category of good effective orbifold atlases} \label{GPreOrb}

Throughout this paper, let $X,Y,Z$ denote paracompact, second countable and Hausdorff spaces. For a connected open subset $U\subseteq X$, an \emph{effective} (also \emph{reduced}) \emph{orbifold chart} (also \emph{uniformizing system}) of dimension $n$ over $U$ consists of
\begin{enumerate}[label=(\arabic*)]
\item a connected open subset $\tU\subseteq\bbR^n$,
\item a finite subgroup $G<\Diff(\tU)$ of smooth automorphisms of $\tU$ which acts on $\tU$ effectively,
\item a continuous, surjective and $G$-invariant map $\pi:\tU\to U$ which induces a homeomorphism between $\tU/G$ (equipped with the quotient topology) and $U$.
\end{enumerate}
For simplicity, in what follows, we refer to effective orbifold charts just as \emph{charts}. $U$ is called a \emph{uniformized set} if it is equipped with a chart $(\tU,G,\pi)$.

\begin{rem}
To facilitate our discussion, we do not exclude the case of $\tU=\varnothing$ which, together with the trivial group and the empty map, defines a chart over $U=\varnothing$.
\end{rem}

For open subsets $U\subseteq V\subseteq X$, an \emph{embedding} $\lam:(\tU,G,\pi)\to(\tV,H,\tau)$ of charts is a smooth embedding $\lam:\tU\to\tV$ that is a lifting of the inclusion $U\,\to/^(->/V$, i.e., such that the square
$$\bfig
\square/->`->`->`^(->/<500,400>[\tU`\tV`U`V;\lam`\pi`\tau`]
\efig$$
is commutative.

\begin{rem} \label{two-embedding-group}
If $\lam,\lam':(\tU,G,\pi)\to(\tV,H,\tau)$ are both embeddings, the effectiveness of the group actions guarantees the existence of a unique $h\in H$ with $\lam=h\circ\lam'$ (see \cite[Proposition A.1]{Moerdijk1997}); hence, an embedding $\lam:(\tU,G,\pi)\to(\tV,H,\tau)$ induces an injective group homomorphism, also denoted by $\lam:G\to H$.
\end{rem}

Given a chart $(\tU,G,\pi)$, a connected open subset $V\subseteq U=\pi(\tU)$ is uniformized by taking a connected component $\tV$ of $\pi^{-1}(V)$ and considering the group $H=\{g\in G\mid g(\tV)\subseteq\tV\}$; then,
\begin{equation} \label{V-H-piV}
(\tV,H,\pi|_{\tV})
\end{equation}
becomes a chart over $V$.\footnote{$H$ acts on $\tV$ effectively whenever $\tV\neq\varnothing$; in the case that $\tV=\varnothing$, one has to generate an effective action on $\tV$ from $H(=G)$, which is precisely the trivial group.} For each $x\in U$ and $\tx\in\pi^{-1}(x)$, the subgroup
$$G_x:=\{g\in G\mid g\cdot\tx=\tx\}$$
is called the \emph{isotropy subgroup} (also \emph{stabilizer subgroup}) at $x$, which is uniquely determined up to conjugacy in $G$.

\begin{defn} \label{atlas}  \cite{Adem2007,Satake1956}
An \emph{$n$-dimensional effective orbifold atlas} on $X$ is a family $\CU=\{(\tU_i,G_i,\pi_i)\mid i\in I\}$ of charts of dimension $n$ on $X$, such that
\begin{enumerate}[label=(\arabic*)]
\item $\{U_i=\pi_i(\tU_i)\mid i\in I\}$ covers $X$, and
\item for each $x\in U_i\cap U_j$, there exists $(\tU_k,G_k,\pi_k)\in\CU$ with $x\in U_k\subseteq U_i\cap U_j$ and embeddings
    $$(\tU_i,G_i,\pi_i)\toleft^{\lam_{ki}}(\tU_k,G_k,\pi_k)\to^{\lam_{kj}}(\tU_j,G_j,\pi_j).$$
\end{enumerate}
\end{defn}

All orbifold atlases considered in this paper are $n$-dimensional and effective, and thus we will just refer to them as \emph{atlases} if no confusion arises. It is easy to observe the following important fact:

\begin{prop} \label{atlas-Cat} {\rm\cite{Tommasini2012}}
An atlas $\CU$ is a small category with its charts as objects and their embeddings as morphisms.
\end{prop}

For atlases $\CU,\CU'$ on $X$, $\CU$ is called a \emph{refinement} of $\CU'$ if, for any $(\tU_i,G_i,\pi_i)\in\CU$, there exists $(\tU'_i,G'_i,\pi'_i)\in\CU'$ with an embedding $\lam:(\tU_i,G_i,\pi_i)\to(\tU'_i,G'_i,\pi'_i)$. Two atlases $\CU$, $\CV$ on $X$ are \emph{equivalent} if they have a common refinement. An equivalence class $[\CU]$ of atlases is an \emph{orbifold structure} on $X$, and the pair $\CX=(X,[\CU])$ is called an \emph{orbifold}.

The following Definition \ref{good-atlas} is motivated by the results in \cite[Subsection 1.2]{Moerdijk1999}. Explicitly, given an atlas $\CU$ on $X$, there is a \emph{good triangulation} $T_{\CU}$ of $X$ associated to $\CU$. Since the isotopy group in the interior of a simplex in $T_{\CU}$ is constant, which is a subgroup of the isotopy group of a vertex in $T_{\CU}$, a ``good atlas'' $\CV$ refining $\CU$ can be constructed as follows: Let $V_{T_{\CU}}$ denote the set of vertices in $T_{\CU}$, which can be assumed to be countable since $X$ is second countable. Then, for a $(p+1)$-tuple $J=(j_0,\dots,j_p)$ in $V_{T_{\CU}}$, one may consider the unique simplex in $T_{\CU}$ with vertices exactly being $(j_0,\dots,j_p)$ whenever it exists, and find its open star neighbourhood $V_J$, which is necessarily contractible. Since there is a chart $(\tU,G,\pi)\in\CU$ with $V_J\subseteq U=\pi(\tU)$, a chart $(\tV_J,H_J,\pi_J)$ over $V_J$ can be obtained as \eqref{V-H-piV} with $\tV_J$ being also contractible, in which way the desired atlas $\CV=\{(\tV_J,H_J,\pi_J)\mid J=(j_0,\dots,j_p),\ j_0,\dots,j_p\in V_{T_{\CU}},\ p\in\bbZ_{\geq 0}\}$ is constructed.

\begin{defn} \label{good-atlas}
A \emph{good atlas} $\CU=\{(\tU_I,G_I,\pi_I)\mid I\in\FVU\}$ on $X$ is an atlas such that
\begin{enumerate}[label=(\arabic*)]
\item \label{good-atlas:countable}
    $V_{\CU}$ is a countable set and $\FVU$ is the set of all finite subsets of $V_{\CU}$;
\item \label{good-atlas:cover}
    $\{U_i=\pi_i(\tU_i)\mid i\in V_{\CU}\}$ is a locally finite open cover of $X$, where $(\tU_i,G_i,\pi_i):=(\tU_{\{i\}},G_{\{i\}},\pi_{\{i\}})$;
\item \label{good-atlas:contractible}
    each $\tU_I$ ($I\in\FVU$) is contractible whenever it is non-empty, and so is $U_I=\pi_I(\tU_I)$;
\item \label{good-atlas:finite-intersection}
    $U_I=U_{I_0}\cap U_{I_1}\cap\dots\cap U_{I_p}$ $(p\geq 0)$ if $I=I_0\cup I_1\cup\dots\cup I_p$, and it is uniformized by $(\tU_I,G_I,\pi_I)\in\CU$;
\item \label{good-atlas:embedding}
    there exists an embedding $\lam_{IJ}:(\tU_I,G_I,\pi_I)\to(\tU_J,G_J,\pi_J)$ in $\CU$ whenever $U_I\subseteq U_J$.
\end{enumerate}
\end{defn}

For a good atlas $\CU$ on $X$, it is easily seen that there exist embeddings
$$\lam_{IJ}:(\tU_I,G_I,\pi_I)\to(\tU_J,G_J,\pi_J)$$
if (certainly, \emph{not} only if) $J\subseteq I\in\FVU$, which are morphisms between charts when $\CU$ is considered as a category (see Proposition \ref{atlas-Cat}). Moreover, $\{U_i=\pi_i(\tU_i)\mid i\in V_{\CU}\}$ is a \emph{good cover} of $X$ in the sense of \cite[Appendix, Section 4]{Petersen2006}.

Obviously, the atlases constructed from good triangulations are always good. Since there exists a good triangulation for any atlas, an orbifold $\CX$ is always equipped with a good atlas (cf. \cite[Proposition 1.2.3 and Corollary 1.2.5]{Moerdijk1999}). A good atlas that does not arise from a good triangulation is presented below:

\begin{exmp}
Let $\CX=[\bbC/\bbZ_a]$ $(a\geq 2)$ be the global quotient orbifold with the complex plane $\bbC$ as its underlying space, where $\bbZ_a$ acts on $\bbC$ by rotating $\dfrac{2\pi}{a}$. There is a good atlas $\CU$ on $\CX$ with exactly one non-empty chart
$$(\bbC,\bbZ_a,\tau),$$
where $\tau$ is a homeomorphism from $\bbC/\bbZ_a$ to $\bbC$. However, $\CU$ cannot be generated by any good triangulation. Indeed, good atlases on $\CX$ generated by good triangulations would contain infinitely many charts, since every triangulation of $\bbC$ has infinitely many vertices and edges.
\end{exmp}

In order to organize good atlases into a category, we now introduce compatible systems which are taken as morphisms in Tommasini's category $\PreOrb$ of effective complex orbifold atlases (see \cite{Tommasini2012}), though with a slight modification to fit into our context. We also refer to \cite{Chen2002,Chen2004} for the origin of this concept.

\begin{defn} \label{compat-sys}
For open subsets $U\subseteq X$ and $V\subseteq Y$ respectively uniformized by $(\tU,G,\pi)$ and $(\tV,H,\tau)$, a \emph{(local) lifting} of a continuous map $f:U\to V$ is a smooth function $\tf:\tU\to\tV$ with
\begin{equation} \label{lift-def}
\tau\circ\tf=f\circ\pi.
\end{equation}
$$\bfig
\square<700,400>[\tU`\tV`U`V;\tf`\pi`\tau`f]
\efig$$
For good atlases $\CU=\{(\tU_I,G_I,\pi_I)\mid I\in\FVU\}$ and $\CV=\{(\tV_K,H_K,\tau_K)\mid K\in\FVV\}$ respectively on $X$ and $Y$, a \emph{compatible system} $\tf:\CU\to\CV$ for a continuous map $f:X\to Y$ consists of
\begin{enumerate}[label=(\arabic*)]
\item \label{compat-sys:functor}
    a functor $\tf:\CU\to\CV$ between atlases (see Proposition \ref{atlas-Cat}), whose underlying map on objects is actually a map $\tf:\FVU\to\FVV$ between index sets, such that
    \begin{equation} \label{compat-sys:union}
    f(\pi_I(\tU_I))\subseteq\tau_{\tf I}(\tV_{\tf I})\quad\text{and}\quad\tf(I_1\cup\dots\cup I_p)=(\tf I_1)\cup\dots\cup(\tf I_p)
    \end{equation}
    for all $I,I_1,\dots,I_p\in\FVU$;
\item \label{compat-sys:lift}
    a family $\{\tf_I\mid I\in\FVU\}$ of liftings of $f|_{U_I}:U_I\to f(U_I)\,\to/^(->/V_{\tf I}$ from $(\tU_I,G_I,\pi_I)$ to $(\tV_{\tf I},H_{\tf I},\tau_{\tf I})$ satisfying
    \begin{equation} \label{compat-sys:lift-embed}
    \tf_J\circ\lam_{IJ}=(\tf\lam_{IJ})\circ\tf_I
    \end{equation}
    for all embeddings $\lam_{IJ}:(\tU_I,G_I,\pi_I)\to(\tU_J,G_J,\pi_J)$ in $\CU$, which can be regarded as a lifting of the obvious identity $f|_{U_J}|_{U_I}=f|_{U_I}$.
    $$\bfig
    \cube|blrb|<1500,900>[\tU_J`\tV_{\tf J}`U_J`V_{\tf J};\hskip 1cm\tf_J`\pi_J`\tau_{\tf J}`f|_{U_J}]%
    (700,300)|alrb|<1400,800>[\tU_I`\tV_{\tf I}`U_I`V_{\tf I};\tf_I`\pi_I`\tau_{\tf I}`f|_{U_I}\hskip 1cm]%
    |lrrr|/<-`<-`<-_)`<-_)/[\lam_{IJ}`\tf\lam_{IJ}``]
    \efig$$
\end{enumerate}
\end{defn}

\begin{rem}
Comparing to the definition of compatible systems in \cite{Tommasini2012}, we additionally require the map $\tf:\FVU\to\FVV$ between index sets to preserve finite unions, which is natural according to our definition of good atlases.
\end{rem}

%

Since compatible systems $\tf:\CU\to\CV$ are functors, it is possible to consider natural transformations between them:

\begin{defn} \label{compat-sys-nat}
Let $\tf^1,\tf^2:\CU\to\CV$ be compatible systems of the same continuous map $f:X\to Y$. A natural transformation $\al:\tf^1\to/=>/\tf^2$ of functors becomes a \emph{natural transformation of compatible systems} if
\begin{equation} \label{compat-sys-nat:lifting}
\tf^2_I=\al_I\circ\tf^1_I
\end{equation}
for all $I\in\FVU$. Explicitly, $\al$ is given by a family
$$\{\al_I:(\tV_{\tf^1 I},H_{\tf^1 I},\tau_{\tf^1 I})\to(\tV_{\tf^2 I},H_{\tf^2 I},\tau_{\tf^2 I})\mid I\in\FVU\}$$
of embeddings in $\CV$, such that the diagram
$$\bfig
\square<800,500>[\tV_{\tf^1 I}`\tV_{\tf^1 J}`\tV_{\tf^2 I}`\tV_{\tf^2 J};\tf^1\lam_{IJ}`\al_I`\al_J`\tf^2\lam_{IJ}]
\qtriangle(-800,0)/->`->`/<800,500>[\tU_I`\tV_{\tf^1 I}`\tV_{\tf^2 I};\tf^1_I`\tf^2_I`]
\ptriangle(800,0)/<-``->/<800,500>[\tV_{\tf^1 J}`\tU_J`\tV_{\tf^2 J};\tf^1_J``\tf^2_J]
\efig$$
is commutative for all embeddings $\lam_{IJ}:(\tU_I,G_I,\pi_I)\to(\tU_J,G_J,\pi_J)$ in $\CU$.
\end{defn}

We are now ready to present the 2-category\footnote{The readers are assumed to be familiar with basic notions of 2-categories; see, e.g. \cite{Borceux1994a,Leinster2004,MacLane1998}.} of good atlases:

\begin{prop} \label{GPreOrb-Cat}
With good atlases as objects, compatible systems as 1-cells and their natural transformations as 2-cells, one obtains a 2-category
$$\GPreOrb.$$
\end{prop}

\begin{proof}
Given compatible systems $\tf:\CU\to\CV$, $\tg:\CV\to\CW$ respectively for continuous maps $f:X\to Y$, $g:Y\to Z$, their composite is given by the functor $\tg\circ\tf:\CU\to\CW$ (which clearly satisfies \eqref{compat-sys:union}) and the family
\begin{equation} \label{composite-lift}
\{(\tg\circ\tf)_I:=\tg_{\tf I}\circ\tf_I\mid I\in\FVU\}
\end{equation}
of liftings of $(g\circ f)|_{U_I}:U_I\to g\circ f(U_I)\,\to/^(->/W_{\tg\tf I}$ satisfying \eqref{compat-sys:lift-embed}; indeed, the diagram
$$\bfig
\iiixii|bblrraa|<1000,500>[\tU_I`\tV_{\tf I}`\tW_{\tg\tf I}`\tU_J`\tV_{\tf J}`\tW_{\tg\tf J};\tf_I`\tg_{\tf I}`\lam_{IJ}`\tf{\lam_{IJ}}`\tg\tf\lam_{IJ}`\tf_J`\tg_{\tf J}]
\square|abbb|/{@{->}@/^1.5em/}```{@{->}@/_1.5em/}/<2000,500>[\tU_I`\tW_{\tg\tf I}`\tU_J`\tW_{\tg\tf J};(\tg\circ\tf)_I```(\tg\circ\tf)_J]
\efig$$
is commutative for all embeddings $\lam_{IJ}:(\tU_I,G_I,\pi_I)\to(\tU_J,G_J,\pi_J)$ in $\CU$, by applying \eqref{compat-sys:lift-embed} to $\tf$ and $\tg$. With the identity compatible system $1_{\CU}$ for the identity map $1_X:X\to X$ given by the identity functor on $\CU$ and the identity liftings $\{(\tidU)_I=1_{\tU_I}:\tU_I\to\tU_I\mid I\in\FVU\}$, one obviously obtains a category $\GPreOrb$ of good atlases and compatible systems.

Note that natural transformations of compatible systems are a special kind of natural transformations of functors. In order to show that $\GPreOrb$ has a 2-category structure with 2-cells given by natural transformations of compatible systems, it suffices to prove that they are closed under (vertical and horizontal) compositions of natural transformations as in the 2-category $\Cat$ of small categories, functors and natural transformations.

Given natural transformations $\tf^1\to/=>/^\al\tf^2\to/=>/^\be\tf^3:\CU\to\CV$ of compatible systems of the same continuous map $f:X\to Y$, the (vertical) composition $\be\al:\tf^1\to/=>/\tf^3$ satisfies \eqref{compat-sys-nat:lifting} since the diagram
$$\bfig
\btriangle<700,500>[\tU_I`\tV_{\tf^1 I}`\tV_{\tf^2 I};\tf^1_I`\tf^2_I`\al_I]
\morphism(0,500)/{@{->}@/^1em/}/<1500,-500>[\tU_I`\tV_{\tf^3 I};\tf^3_I]
\morphism(700,0)|b|<800,0>[\tV_{\tf^2 I}`\tV_{\tf^3 I};\be_I]
\efig$$
is commutative for all $I\in\FVU$, by applying \eqref{compat-sys-nat:lifting} to $\al$ and $\be$. Hence, $\be\al$ is a natural transformation of compatible systems.

Given natural transformations $\al:\tf^1\to/=>/\tf^2:\CU\to\CV$ and $\be:\tg^1\to/=>/\tg^2:\CV\to\CW$ of compatible systems, respectively of the continuous maps $f:X\to Y$ and $g:Y\to Z$, the (horizontal) composition $\be\circ\al:\tg^1\circ\tf^1\to/=>/\tg^2\circ\tf^2$ satisfies \eqref{compat-sys-nat:lifting} since the diagram
$$\bfig
\qtriangle<800,400>[\tU_I`\tV_{\tf^1 I}`\tV_{\tf^2 I};\tf^1_I`\tf^2_I`\al_I]
\qtriangle(800,0)<1000,400>[\tV_{\tf^1 I}`\tW_{\tg^1\tf^1 I}`\tW_{\tg^2\tf^1 I};\tg^1_{\tf^1 I}`\tg^2_{\tf^1 I}`\be_{\tf^1 I}]
\qtriangle(800,-400)/`->`->/<1000,400>[\tV_{\tf^2 I}`\tW_{\tg^2\tf^1 I}`\tW_{\tg^2\tf^2 I};`\tg^2_{\tf^2 I}`\tg^2\al_I]
\qtriangle(0,-400)/{@{->}@/^2.5em/}`{@{->}@/_3em/}`{@{->}@/^4em/}/<1800,800>[\tU_I`\tW_{\tg^1\tf^1 I}`\tW_{\tg^2\tf^2 I};(\tg^1\circ\tf^1)_I`(\tg^2\circ\tf^2)_I`(\be\circ\al)_I]
\efig$$
is commutative for all $I\in\FVU$, by applying \eqref{compat-sys-nat:lifting} to $\al$, $\be$ and \eqref{compat-sys:lift-embed} to the embedding $\al_I:(\tV_{\tf^1 I},H_{\tf^1 I},\tau_{\tf^1 I})\to(\tV_{\tf^2 I},H_{\tf^2 I},\tau_{\tf^2 I})$. Hence, $\be\circ\al$ is a natural transformation of compatible systems, completing the proof.
\end{proof}

\section{The 2-category of spark complexes} \label{SpCx}

A \emph{homological spark complex} \cite{Hao2009,Harvey2006,Harvey2003}, or \emph{spark complex} for short, is a triple $(\sF^*,\sE^*,\sI^*)$ of cochain complexes\footnote{All cochain complexes considered in this paper are of abelian groups and bounded below.}, such that
\begin{enumerate}[label=(\arabic*)]
\item $\sI^*$ and $\sE^*$ are both subcomplexes of $\sF^*$,
\item $\sI^k\cap\sE^k=\{0\}$ for $k>0$, $\sF^k=\sE^k=\sI^k=\{0\}$ for $k<0$, and
\item $H^*(\sE^*)\cong H^*(\sF^*)$.
\end{enumerate}
A \emph{spark} of degree $k$ is an element $a\in\sF^k$ satisfying the \emph{spark equation}
\begin{equation} \label{spark-eq}
da=e-r
\end{equation}
for some $e\in\sE^{k+1}$, $r\in\sI^{k+1}$ (which are necessarily unique, see \cite[Lemma 1.2]{Harvey2006}). Sparks $a,a'$ of degree $k$ are \emph{equivalent} if
\begin{equation} \label{spark-equiv}
a-a'=db+s
\end{equation}
for some $b\in\sF^{k-1}$, $s\in\sI^k$. Given a spark $a\in\sF^k$, the equivalence class containing $a$ is denoted by $[a]$, called a \emph{spark character} (also \emph{spark class}). We write $\hH^k(\sF^*,\sE^*,\sI^*)$ for the group of spark characters of degree $k$, and
$$\hH(\sF^*,\sE^*,\sI^*)=\bigoplus_{k\in\bbZ}\hH^k(\sF^*,\sE^*,\sI^*)$$
for the graded abelian group of all spark characters on a spark complex $(\sF^*,\sE^*,\sI^*)$, where $\bbZ$ is the discrete set of integers.


A \emph{spark homomorphism} $f:(\sF^*,\sE^*,\sI^*)\to(\oF^*,\oE^*,\oI^*)$ of spark complexes is a cochain map $f:\sF^*\to\oF^*$ such that the diagram
$$\bfig
\iiixii|aalrrbb|/^(->`<-^)`->`->`->`^(->`<-^)/<500,400>[\sI^*~`\sF^*`~\sE^*`\oI^*~`\oF^*`~\oE^*;``f|_{\sI^*}`f`f|_{\sE^*}``]
\efig$$
is commutative or, equivalently, such that $f(\sI^*)$ and $f(\sE^*)$ are subcomplexes of $\oI^*$ and $\oE^*$, respectively.

Similar to the obvious way of organizing cochain complexes into a 2-category, we define homotopies between spark homomorphisms as follows:

\begin{defn} \label{spark-homo}
A \emph{spark homotopy} $\Phi:f\to/=>/g:(\sF^*,\sE^*,\sI^*)\to(\oF^*,\oE^*,\oI^*)$ between spark homomorphisms is a cochain homotopy $\Phi:f\to/=>/g:\sF^*\to\oF^*$ which vanishes on $\sE^*$ and restricts to a cochain homotopy $\Phi|_{\sI^*}:f|_{\sI^*}\to/=>/g|_{\sI^*}:\sI^*\to\oI^*$; that is, a graded group homomorphism $\Phi:\sF^*\to\oF^{*-1}$ with
$$\od\circ\Phi+\Phi\circ d=g-f,\quad\Phi|_{\sE^*}=0$$
and $\Phi(\sI^*)$ being a graded subgroup of $\oI^{*-1}$, where $d$ and $\od$ are differentials on $\sF^*$ and $\oF^*$, respectively. $f$ is \emph{homotopic} to $g$, denoted by $f\sim g$, if there exists a spark homotopy $\Phi:f\to/=>/g$.

Moreover, a \emph{homotopy} $\Ga:\Phi\toRr\Psi$ between spark homotopies is precisely a homotopy of cochain homotopies subject to the same restrictions; hence, $\Ga$ is a graded group homomorphism $\Ga:\sF^*\to\oF^{*-2}$ with
$$\od\circ\Ga-\Ga\circ d=\Psi-\Phi,\quad\Ga|_{\sE^*}=0$$
and $\Ga(\sI^*)$ being a graded subgroup of $\oI^{*-2}$. $\Phi$ is \emph{homotopic} to $\Psi$, denoted by $\Phi\sim\Psi$, if there exists a homotopy $\Ga:\Phi\toRr\Psi$; the induced equivalence class, i.e., the \emph{spark homotopy class}, of a spark homotopy $\Phi$, is denoted by $[\Phi]$.
\end{defn}


\begin{prop} \label{SpCx-Cat}
With spark complexes as objects, spark homomorphisms as 1-cells and spark homotopy classes as 2-cells, one obtains a 2-category
$$\SpCx.$$
\end{prop}

\begin{proof}
Based on the structure of the 2-category of cochain complexes, the only non-trivial part of this proof is the closedness of spark homotopy classes under (vertical and horizontal) compositions of cochain homotopies.

Given spark homotopies $f\to/=>/^{\Phi}g\to/=>/^{\Psi}h:(\sF^*,\sE^*,\sI^*)\to(\oF^*,\oE^*,\oI^*)$, their vertical composition $\Psi+\Phi:f\to/=>/h$ is clearly a spark homotopy. If $\Phi\sim\Phi'$ and $\Psi\sim\Psi'$, we find $\Ga:\Phi\toRr\Phi'$ and $\Ga':\Psi\toRr\Psi'$, then it is easy to see that $\Ga+\Ga':\Psi+\Phi\toRr\Psi'+\Phi'$ is a homotopy of spark homotopies, whence $\Phi+\Psi\sim\Phi'+\Psi'$.

Given spark homotopies $\Phi:f\to/=>/g:(\sF_1^*,\sE_1^*,\sI_1^*)\to(\sF_2^*,\sE_2^*,\sI_2^*)$ and $\Psi:h\to/=>/k:(\sF_2^*,\sE_2^*,\sI_2^*)\to(\sF_3^*,\sE_3^*,\sI_3^*)$, the horizontal composition of the corresponding homotopy classes
$$[\Psi\circ\Phi]:=[\Psi\circ f+k\circ\Phi]=[\Psi\circ g+h\circ\Phi]:h\circ f\to/=>/k\circ g$$
is clearly a spark homotopy class. If $\Phi\sim\Phi'$ and $\Psi\sim\Psi'$, we find $\Ga:\Phi\toRr\Phi'$ and $\Ga':\Psi\toRr\Psi'$, then it is straightforward to check that
$$\Ga'\circ g+h\circ\Ga:(\Psi\circ g+h\circ\Phi)\toRr(\Psi'\circ g+h\circ\Phi')$$
is a homotopy of spark homotopies, showing that $\Psi\circ\Phi\sim\Psi'\circ\Phi'$.
\end{proof}

It is clear that every 2-cell in $\SpCx$ is an isomorphism, and moreover:

\begin{prop} \label{Char-functor}
A spark homomorphism $f:(\sF^*,\sE^*,\sI^*)\to(\oF^*,\oE^*,\oI^*)$ induces a homomorphism
$$f_*:\hH(\sF^*,\sE^*,\sI^*)\to\hH(\oF^*,\oE^*,\oI^*),\quad [a]\mapsto[fa]$$
of the associated graded abelian groups of spark characters, and
$$f_*=g_*$$
if there exists a spark homotopy $\Phi:f\to/=>/g:(\sF^*,\sE^*,\sI^*)\to(\oF^*,\oE^*,\oI^*)$.
\end{prop}

\begin{proof}
$f_*$ is well-defined since $a=db+s$ for some $b\in\sF^{k-1}$, $s\in\sI^k$ implies
$$fa=f(db+s)=\od(fb)+fs$$
with $fs\in\oI^k$. If $\Phi:f\to/=>/g$ is a spark homotopy, for any spark $a\in\sF^*$ with $da=e-r$, $e\in\sE^*$, $r\in\sI^*$ one has
$$ga-fa=\od(\Phi a)+\Phi(da)=\od(\Phi a)+\Phi e-\Phi r=\od(\Phi a)-\Phi r,$$
then it follows from $\Phi r\in\oI^*$ that $[ga]=[fa]$. Hence $f_*=g_*$.
\end{proof}

Let $\GAb$ denote the category of $(\bbZ\text{-})$graded abelian groups and their homomorphisms. If one considers $\GAb$ as a 2-category only equipped with trivial 2-cells, then Proposition \ref{Char-functor} in fact defines a 2-functor
\begin{equation} \label{Char}
\Char:\SpCx\to\GAb
\end{equation}
that sends each spark homomorphism $f:(\sF^*,\sE^*,\sI^*)\to(\oF^*,\oE^*,\oI^*)$ to the homomorphism
$$f_*:\hH(\sF^*,\sE^*,\sI^*)\to\hH(\oF^*,\oE^*,\oI^*)$$
of graded abelian groups, and sends each spark homotopy class $[\Phi]:f\to/=>/g$ to the identity 2-cell on $f_*=g_*$.

\section{Spark complexes on good atlases} \label{Spark-good-atlas}

The aim of this section is to construct a spark complex on every good atlas. Since we work on \emph{orbifold atlases} instead of orbifolds (i.e., equivalence classes of orbifold atlases), tangent and cotangent bundles of orbifolds will not be necessary for our discussion.

Let $\CU=\{(\tU_I,G_I,\pi_I)\mid I\in\FVU\}$ be a good atlas. Recall that a \emph{differential $q$-form} \cite{Tu2011} on $\tU_I$ assigns to each $x\in\tU_I$ an alternating $q$-linear map on its tangent space $T_x\tU_I$; that is, a $q$-linear map $\om_I:(T_x\tU_I)^q\to\bbR$ satisfying
$$\om_I(v_{\si(1)},\dots,v_{\si(q)})=(\sgn\si)\om_I(v_1,\dots,v_q)$$
for all $v_1,\dots,v_q\in T_x\tU_I$ and permutations $\si$ of the set $\{1,\dots,k\}$. Denoting by $\Om^q(\tU_I)$ the set of differential $q$-forms on $\tU_I$, a \emph{$q$-form} \cite{Satake1956} on the orbifold atlas $\CU$ is a family
$$\{\om_I\in\Om^q(\tU_I)\mid I\in\FVU\}$$
of differential $q$-forms such that
\begin{enumerate}[label=(\arabic*)]
\item each $\om_I$ is $G_I$-invariant, and
\item $\lam_{IJ}^*(\om_J)=\om_I$ for any embedding $\lam_{IJ}:(\tU_I,G_I,\pi_I)\to(\tU_J,G_J,\pi_J)$, where $\lam_{IJ}^*(\om_J)$ is the pullback\footnote{The pullback of differential forms should be carefully distinguished from the same terminology in category theory.} of $\om_J$ along $\lam_{IJ}$ given by
    $$\lam_{IJ}^*(\om_J)(v_1,\dots,v_q)=\om_J((\lam_{IJ})_{*,x}v_1,\dots,(\lam_{IJ})_{*,x} v_q)$$
    for all $x\in\tU_I$ and $v_1,\dots,v_q\in T_x\tU_I$, and the linear map $(\lam_{IJ})_{*,x}:T_x\tU_I\to T_{\lam_{IJ}x}\tU_J$ of tangent spaces is the differential of $\lam_{IJ}$ at $x$.
\end{enumerate}

We write $\oOm^q(\CU)$ for the set of $q$-forms on $\CU$. Obviously, the exterior derivative $d$ acts on $\oOm^*(\CU)$, and the usual wedge product ``$\wedge$'' can be restricted to $\oOm^*(\CU)$. Thus, we obtain a DGA (differential graded algebra)
\begin{equation} \label{oOm-def}
(\oOm^*(\CU),d,\wedge).
\end{equation}

Now we are ready to present the construction of a spark complex
$$\oS_{\CU}=(\oF^*_{\CU},\oE^*_{\CU},\oI^*_{\CU})$$
associated to the given good atlas $\CU$. Let
$$\MVU$$
denote the \emph{free monoid} on the set $\FVU$, whose elements are \emph{strings} (or \emph{words})
$$\FI=I_0\dots I_p$$
consisting of elements of $\FVU$. For $\FI,\FJ\in\MVU$, we write $\FI\FJ$ for the \emph{string concatenation} of $\FI$ and $\FJ$, i.e., the monoid multiplication of $\MVU$, whose unit is obviously given by the empty string.

\begin{rem} \label{FI-notation}
In order to facilitate our discussions below, we introduce here several notations for each string $\FI=I_0\dots I_p\in\MVU$:
\begin{enumerate}[label=(\arabic*)]
\item \label{FI-notation:cup}
    $\cup\FI:=I_0\cup\dots\cup I_p\in\FVU$ is the union of elements in $\FI$;
\item \label{FI-notation:k}
    $\FI_{\hk}:=I_0\dots\hI_k\dots I_p:=I_0\dots I_{k-1}I_{k+1}\dots I_p$ refers to the removal of the $(k+1)$th element $I_k$ of the string;
\item \label{FI-notation:jk}
    $\FI_{j\mathrel{\leftrightarrow}k}$ $(j\neq k)$ refers to the switching of the positions of $I_j$ and $I_k$;
\item \label{FI-notation:fI}
    Each map $\tf:\FVU\to\FVV$ sends $\FI$ to $\tf\FI:=(\tf I_0)\dots(\tf I_p)\in\MVV$;
\item \label{FI-notation:MpVU}
    $\MpVU$ denotes the subset of $\MVU$ consisting of strings of length $p+1$.
\end{enumerate}
\end{rem}

For each $\FI\in\MVU$, it is natural to define
\begin{equation} \label{tU-FI-def}
(\tU_{\FI},G_{\FI},\pi_{\FI}):=(\tU_{\cup\FI},G_{\cup\FI},\pi_{\cup\FI})\in\CU,
\end{equation}
although one has to be careful that different strings in $\MVU$ may correspond to the same chart in $\CU$; in particular, if $\FI=I_0\dots I_p$, then
$$U_{\cup\FI}=U_{I_0}\cap\dots\cap U_{I_p}$$
is uniformized by $(\tU_{\FI},G_{\FI},\pi_{\FI})$ (see Definition \ref{good-atlas}\ref{good-atlas:finite-intersection}). With
$$\Om^q(\tU_{\FI})^{G_{\FI}}:=\{\om_{\FI}\in\Om^q(\tU_{\FI})\mid \om_{\FI}\ \text{is}\ G_{\FI}\text{-invariant}\}\subseteq\Om^q(\tU_{\FI})$$
denoting the set of $G_{\FI}$-invariant differential $q$-forms on $\tU_{\FI}$, we define
\begin{equation} \label{oCpq-def}
\oC^p(\CU,\Om^q):=\Big\{(\om_{\FI})\in\prod_{\FI\in\MpVU}\Om^q(\tU_{\FI})^{G_{\FI}}\mathrel{\Big|}\om_{\FI}=-\om_{\FI_{j\mathrel{\leftrightarrow}k}}\ \text{whenever}\ 0\leq j,k\leq p\Big\}
\end{equation}
for all $p,q\in\bbZ_{\geq 0}$, and $\oC^p(\CU,\Om^q)=\{0\}$ otherwise. Then
$$\oC^*(\CU,\Om^*)=\bigoplus_{p,q\in\bbZ}\oC^p(\CU,\Om^q)$$
becomes a double complex $(\oC^*(\CU,\Om^*),d,\de)$ with $d$ the exterior derivative on $\oC^*(\CU,\Om^*)$ and
\renewcommand\arraystretch{1.5}
\begin{equation} \label{delta-def}
\begin{array}{rccc}
\de: & \oC^p(\CU,\Om^q) & \to & \oC^{p+1}(\CU,\Om^q)\\
& \om & \mapsto & \de\om=((\de\om)_{\FI})\\
& & & (\de\om)_{\FI}=\dsum\limits_{k=0}^{p+1}(-1)^k\om_{\FI_{\hk}}\big|_{\tU_{\FI}}.
\end{array}
\end{equation}
Indeed, the squares
$$\bfig
\square/->`<-`<-`->/<1000,400>[\oC^p(\CU,\Om^{q+1})`\oC^{p+1}(\CU,\Om^{q+1})`\oC^p(\CU,\Om^q)`\oC^{p+1}(\CU,\Om^q);\de`d`d`\de]
\morphism(-600,400)<600,0>[\dots`\oC^p(\CU,\Om^{q+1});]
\morphism(-600,0)<600,0>[\dots`\oC^p(\CU,\Om^q);]
\morphism(1000,400)<600,0>[\oC^{p+1}(\CU,\Om^{q+1})`\dots;]
\morphism(1000,0)<600,0>[\oC^{p+1}(\CU,\Om^q)`\dots;]
\morphism(0,400)<0,350>[\oC^p(\CU,\Om^{q+1})`\vdots;]
\morphism(1000,400)<0,350>[\oC^{p+1}(\CU,\Om^{q+1})`\vdots;]
\morphism(0,-400)<0,400>[\vdots`\oC^p(\CU,\Om^q);]
\morphism(1000,-400)<0,400>[\vdots`\oC^{p+1}(\CU,\Om^q);]
\efig$$
are obviously commutative and, moreover, $d^2=0$ and $\de^2=0$.

\begin{rem} \label{restrict-indep-embed}
The restriction $\om_{\FI_{\hk}}\big|_{\tU_{\FI}}=\lam^*_{\FI,\FI_{\hk}}(\om_{\FI_{\hk}})$ given in \eqref{delta-def} is independent of the choice of the embedding
$$\lam_{\FI,\FI_{\hk}}:\tU_{\FI}\to\tU_{\FI_{\hk}}$$
since $\om_{\FI_{\hk}}$ is $G_{\FI_{\hk}}$-invariant (see Remark \ref{two-embedding-group}). For the same reason, in what follows we do not specify the embeddings while taking restrictions of $q$-forms in $\oC^p(\CU,\Om^q)$.
\end{rem}

Let $(\Tot(\oC^*(\CU,\Om^*)),\oD)$ denote the total complex of $\oC^*(\CU,\Om^*))$, with
$$\Tot(\oC^*(\CU,\Om^*))^k=\bigoplus_{p+q=k}\oC^p(\CU,\Om^q)$$
and
\begin{equation} \label{oD-def}
\oD=\de+(-1)^p d
\end{equation}
on $\oC^p(\CU,\Om^q)$. Obviously,
$$(\oOm^*(\CU),d)=\big(\ker\de|_{\oC^0(\CU,\Om^*)},d\big)\quad\text{and}\quad(\oC^*(\CU,\bbZ),\de)$$
are both subcomplexes of $(\Tot(\oC^*(\CU,\Om^*)),\oD)$. Therefore:

\begin{thm} \label{oSU-def}
For every good atlas $\CU$,
$$\oS_{\CU}=(\oF^*_{\CU},\oE^*_{\CU},\oI^*_{\CU}):=(\Tot(\oC^*(\CU,\Om^*)),\oOm^*(\CU),\oC^*(\CU,\bbZ))$$
is a spark complex.
\end{thm}

\begin{proof}
First, it is easy to see that $\oE^*_{\CU}\cap\oI^*_{\CU}=\{0\}$ for $k>0$. Second, $H^*(\oF^*_{\CU},\oD)\cong H^*(\oE^*_{\CU},d)$ since the rows of the double complex $\oC^*(\CU,\Om^*)$ are exact by partition of unity (see \cite[Proposition 49]{Borzellino2008}), which implies that $H^*(\oF^*_{\CU},\oD)$ is isomorphic to the cohomology of the initial column of the double complex $\oC^*(\CU,\Om^*)$ (see the argument below the proof of \cite[Proposition 8.8]{Bott1982}). The conclusion thus follows.
\end{proof}

\section{1-functoriality of spark complexes on good atlases} \label{Functor}

In this section, we show that the assignment $\CU\mapsto\oS_{\CU}$ defined in Theorem \ref{oSU-def} gives rise to a contravariant functor from the category $\GPreOrb$ of good atlases to the category $\SpCx$ of spark complexes. The following lemma is useful for later calculations:

\begin{lem} \label{lift-indep-embed}
Let $\tf:\CU\to\CV$ be a compatible system and $\om\in\oC^p(\CV,\Om^q)$. Then
$$(\tf^*_J(\om_{\tf J}))|_{\tU_I}=\tf^*_I(\om_{\tf J}|_{\tV_{\tf I}})$$
whenever $J\subseteq I\in\FVU$.
\end{lem}

\begin{proof}
By Remark \ref{restrict-indep-embed},
$$\om_J|_{\tU_I}=\lam^*_{IJ}(\om_J)$$
for any embedding $\lam_{IJ}:(\tU_I,G_I,\pi_I)\to(\tU_J,G_J,\pi_J)$ of charts in $\CU$, whose existence is guaranteed by $J\subseteq I$. Hence
$$(\tf^*_J(\om_{\tf J}))|_{\tU_I}=\lam^*_{IJ}\tf^*_J(\om_{\tf J})=\tf^*_I(\tf\lam_{IJ})^*(\om_{\tf J})=\tf^*_I(\om_{\tf J}|_{\tV_{\tf I}}),$$
where the second equality follows from Eq. \eqref{compat-sys:lift-embed} in Definition \ref{compat-sys}, and the third equality holds since $\tf\lam_{IJ}:(\tV_{\tf I},H_{\tf I},\tau_{\tf I})\to(\tV_{\tf J},H_{\tf J},\tau_{\tf J})$ is an embedding of charts in $\CV$.
\end{proof}

\begin{prop} \label{compat-sys-induce-spark}
Each compatible system $\tf:\CU\to\CV$ induces a spark homomorphism
$$\utf:\oS_{\CV}\to\oS_{\CU}.$$
\end{prop}

\begin{proof}
For every $\om\in\oC^p(\CV,\Om^q)$, define
\begin{equation} \label{utf-def}
\utf\om\in\oC^p(\CU,\Om^q)\quad\text{with}\quad(\utf\om)_{\FI}=\tf^*_{\FI}(\om_{\tf\FI});
\end{equation}
that is, for every $\FI=I_0\dots I_p\in\MVU$, $(\utf\om)_{\FI}$ is the pullback of $\om_{\tf\FI}=\om_{(\tf I_0)\dots(\tf I_p)}$ (see Remark \ref{FI-notation}\ref{FI-notation:fI}) along the lifting (cf. Eq. \eqref{compat-sys:union})
$$\tf_{\FI}:=\tf_{\cup\FI}=\tf_{I_0\cup\dots\cup I_p}:\tU_{\FI}\to\tV_{\tf\FI}=\tV_{(\tf I_0)\cup\dots\cup(\tf I_p)}=\tV_{\tf(\cup\FI)}.$$
We show that $\utf:\oS_{\CV}\to\oS_{\CU}$ is a spark homomorphism.

First, $\utf:\Tot(\oC^*(\CV,\Om^*))\to\Tot(\oC^*(\CU,\Om^*))$ is a cochain map. Since $\utf$ obviously commutes with $d$, it suffices to show that $\utf$ commutes with $\de$. Indeed, for any $\om\in\oC^p(\CV,\Om^q)$ and $\FI\in\MppVU$,
\begin{align*}
(\de\utf\om)_{\FI}&=\sum\limits_{k=0}^{p+1}(-1)^k(\utf\om)_{\FI_{\hk}}\big|_{\tU_{\FI}}&(\text{Eq.}\ \eqref{delta-def})\\
&=\sum\limits_{k=0}^{p+1}(-1)^k(\tf^*_{\FI_{\hk}}(\om_{\tf\FI_{\hk}}))\big|_{\tU_{\FI}}&(\text{Eq.}\ \eqref{utf-def})\\
&=\sum\limits_{k=0}^{p+1}(-1)^k\tf^*_{\FI}(\om_{\tf\FI_{\hk}}\big|_{\tV_{\tf\FI}})&(\text{Lemma \ref{lift-indep-embed}})\\
&=\tf^*_{\FI}\Big(\sum\limits_{k=0}^{p+1}(-1)^k\om_{(\tf\FI)_{\hk}}\big|_{\tV_{\tf\FI}}\Big)\\
&=\tf^*_{\FI}(\de\om)_{\tf\FI}&(\text{Eq.}\ \eqref{delta-def})\\
&=(\utf\de\om)_{\FI}.&(\text{Eq.}\ \eqref{utf-def})
\end{align*}

Second, $\utf$ clearly maps $\oOm^*(\CV)$ and $\oC^*(\CV,\bbZ)$ into $\oOm^*(\CU)$ and $\oC^*(\CU,\bbZ)$, respectively. Hence $\utf:\oS_{\CV}\to\oS_{\CU}$ is a spark homomorphism.
\end{proof}

\begin{prop} \label{compat-sys-induce-spark-comp}
For compatible systems $\tf:\CU\to\CV$ and $\tg:\CV\to\CW$,
$$\underline{\tg\circ\tf}=\utf\circ\utg:\oS_{\CW}\to\oS_{\CU}.$$
\end{prop}

\begin{proof}
From Proposition \ref{GPreOrb-Cat} we see that the compatible system $\tg\circ\tf:\CU\to\CW$ is given by the composite functor $\tg\circ\tf:\CU\to\CW$ and the family $\{(\tg\circ\tf)_I=\tg_{\tf I}\circ\tf_I\mid I\in\FVU\}$ of liftings. Hence, for any $\om\in\oC^p(\CW,\Om^q)$ and $\FI\in\MpVU$,
\begin{align*}
(\underline{\tg\circ\tf}\om)_{\FI}&=(\tg\circ\tf)^*_{\FI}(\om_{\tg\tf\FI})&(\text{Eq.}\ \eqref{utf-def})\\
&=(\tg_{\tf\FI}\circ\tf_{\FI})^*(\om_{\tg\tf\FI})&(\text{Eq.}\ \eqref{composite-lift})\\
&=\tf^*_{\FI}\circ\tg^*_{\tf\FI}(\om_{\tg\tf\FI})\\
&=\tf^*_{\FI}(\utg\om)_{\tf\FI}&(\text{Eq.}\ \eqref{utf-def})\\
&=(\utf\circ\utg\om)_{\FI},&(\text{Eq.}\ \eqref{utf-def})
\end{align*}
which completes the proof.
\end{proof}

As an immediate consequence of Propositions \ref{compat-sys-induce-spark} and \ref{compat-sys-induce-spark-comp}, we obtain a functor
\begin{equation} \label{oS}
\oS:\GPreOrb^{\op}\to\SpCx
\end{equation}
at the level of 1-categories that sends each compatible system $\tf:\CU\to\CV$ to the spark homomorphism $\utf:\oS_{\CV}\to\oS_{\CU}$.

\section{2-functoriality of spark complexes on good atlases} \label{2-functor}

The functor $\oS$ obtained in the above section is in fact 2-functorial, whose mapping on 2-cells is given below:

\begin{prop} \label{nat-induce-homotopy}
Each natural transformation $\al:\tf^1\to/=>/\tf^2:\CU\to\CV$ of compatible systems induces a spark homotopy
$$\ual:\utf^1\to/=>/\utf^2:\oS_{\CV}\to\oS_{\CU}.$$
\end{prop}

\begin{proof}
For brevity of notations, we write
$$I^m:=\tf^m I\quad\text{and}\quad\FI^m:=I^m_0\dots I^m_p=\tf^m\FI=(\tf^m I_0)\dots(\tf^m I_p)$$
for all $I\in\FVU$, $\FI=I_0\dots I_p\in\MVU$ and $m=1,2$. We also refer to charts $(\tU_I,G_I,\pi_I)\in\CU$ and $(\tV_K,H_K,\tau_K)\in\CV$ just as $\tU_I$ and $\tV_K$, respectively.

By Definition \ref{compat-sys-nat}, there is a family $\{\al_{\FI}:\tV_{\FI^1}\to\tV_{\FI^2}\mid\FI\in\MVU\}$ of embeddings in $\CV$ satisfying
\begin{equation} \label{al-FI}
\tf^2_{\FI}=\al_{\FI}\circ\tf^1_{\FI}
\end{equation}
for all $\FI\in\MVU$. In particular, for any $\FI=I_0\dots I_p\in\MpVU$, one has $V_{\FI^1}\subseteq V_{\FI^2}$ and, consequently, by Definition \ref{good-atlas}\ref{good-atlas:finite-intersection},
\begin{equation} \label{FI1-12j-FI2}
V_{\FI^1}=V_{I^1_0\dots I^1_p}=V_{I^1_0\dots I^1_j I^1_j\dots I^1_p}\subseteq V_{I^1_0\dots I^1_j I^2_j\dots I^2_p}\subseteq V_{I^2_0\dots I^2_j I^2_j\dots I^2_p}=V_{I^2_0\dots I^2_p}=V_{\FI^2}
\end{equation}
whenever $0\leq j\leq p$. Let us denote
$$\FI^{1,2}_j=I^1_0\dots I^1_j I^2_j\dots I^2_p\in\MppVU,$$
then \eqref{FI1-12j-FI2} guarantees the existence of embeddings
\begin{equation} \label{tV-FI12-embed}
\tV_{\FI^1}\to\tV_{\FI^{1,2}_j}\to\tV_{\FI^2}
\end{equation}
in $\CV$ by Definition \ref{good-atlas}\ref{good-atlas:embedding}. Now, for every $\om\in\oC^p(\CV,\Om^q)$, define
\begin{equation} \label{ual-def}
\ual\om\in\oC^{p-1}(\CU,\Om^q)\quad\text{with}\quad(\ual\om)_{\FI}=\sum\limits_{j=0}^{p-1}(-1)^j\tf^{1,*}_{\FI}\big(\om_{\FI^{1,2}_j}\big|_{\tV_{\FI^1}}\big).
\end{equation}
We show that $\ual$ defines a spark homotopy $\utf^1\to/=>/\utf^2$.

First, $\oD\circ\ual+\ual\circ\oD=\utf^2-\utf^1$. For any $\om\in\oC^p(\CV,\Om^q)$ and $\FI=I_0\dots I_p\in\MpVU$, denote
\begin{align*}
&\FI^{1,2}_{j,\hk}:=I^1_0\dots I^1_j I^2_j\dots\hI^2_k\dots I^2_p\in\MpVU\ \text{if}\ 0\leq j\leq k\leq p,\ \text{and}\\
&\FI^{1,2}_{\hk,j}:=I^1_0\dots\hI^1_k\dots I^1_j I^2_j\dots I^2_p\in\MpVU\ \text{if}\ 0\leq k\leq j\leq p.
\end{align*}
Then
\begin{align*}
(\de\ual\om)_{\FI}&=\sum\limits_{k=0}^p(-1)^k(\ual\om)_{\FI_{\hk}}\big|_{\tU_{\FI}}&(\text{Eq.}\ \eqref{delta-def})\\
&=\sum\limits_{k=0}^p(-1)^k\Big[\sum\limits_{j=0}^{p-1}(-1)^j\tf^{1,*}_{\FI_{\hk}}\big(\om_{(\FI_{\hk})^{1,2}_j}\big|_{\tV_{\FI^1_{\hk}}}\big)\Big]\Big|_{\tU_{\FI}}&(\text{Eq.}\ \eqref{ual-def})\\
&=\sum\limits_{k=0}^p(-1)^k\tf^{1,*}_{\FI}\Big[\sum\limits_{j=0}^{p-1}(-1)^j\om_{(\FI_{\hk})^{1,2}_j}\big|_{\tV_{\FI^1_{\hk}}}\big|_{\tV_{\FI^1}}\Big]&(\text{Lemma \ref{lift-indep-embed}})\\
&=\sum\limits_{k=0}^p(-1)^k\tf^{1,*}_{\FI}\Big[\sum\limits_{j=0}^{p-1}(-1)^j\om_{(\FI_{\hk})^{1,2}_j}\big|_{\tV_{\FI^1}}\Big]&(\text{Remark \ref{restrict-indep-embed}})\\
&=\sum\limits_{k=0}^p(-1)^k\tf^{1,*}_{\FI}\Big[\sum\limits_{j=0}^{k-1}(-1)^j\om_{\FI^{1,2}_{j,\hk}}\big|_{\tV_{\FI^1}}+
\sum\limits_{j=k+1}^p(-1)^{j+1}\om_{\FI^{1,2}_{\hk,j}}\big|_{\tV_{\FI^1}}\Big]
\end{align*}
and
\begin{align*}
(\ual\de\om)_{\FI}&=\sum\limits_{j=0}^p(-1)^j\tf^{1,*}_{\FI}\big((\de\om)_{\FI^{1,2}_j}\big|_{\tV_{\FI^1}}\big)&(\text{Eq.}\ \eqref{ual-def})\\
&=\sum\limits_{j=0}^p(-1)^j\tf^{1,*}_{\FI}\Big[\sum\limits_{k=0}^{p+1}(-1)^k\om_{({\FI^{1,2}_j})_{\hk}}\big|_{\tV_{\FI^{1,2}_j}}\big|_{\tV_{\FI^1}}\Big]&(\text{Eq.}\ \eqref{delta-def})\\
&=\sum\limits_{j=0}^p(-1)^j\tf^{1,*}_{\FI}\Big[\sum\limits_{k=0}^{p+1}(-1)^k\om_{({\FI^{1,2}_j})_{\hk}}\big|_{\tV_{\FI^1}}\Big]&(\text{Remark \ref{restrict-indep-embed}})\\
&=\sum\limits_{j=0}^p(-1)^j\tf^{1,*}_{\FI}\Big[\sum\limits_{k=0}^j(-1)^k\om_{\FI^{1,2}_{\hk,j}}\big|_{\tV_{\FI^1}}+
\sum\limits_{k=j}^p(-1)^{k+1}\om_{\FI^{1,2}_{j,\hk}}\big|_{\tV_{\FI^1}}\Big]\\
&=\sum\limits_{k=0}^p(-1)^k\tf^{1,*}_{\FI}\Big[\sum\limits_{j=0}^k(-1)^{j+1}\om_{\FI^{1,2}_{j,\hk}}\big|_{\tV_{\FI^1}}+
\sum\limits_{j=k}^p(-1)^j\om_{\FI^{1,2}_{\hk,j}}\big|_{\tV_{\FI^1}}\Big],
\end{align*}
where Remark \ref{restrict-indep-embed} is applicable to the computations since one has the embeddings \eqref{tV-FI12-embed}. Since obviously $d\ual\om=\ual d\om$, it follows from \eqref{oD-def} that
\begin{align*}
(\oD\ual\om)_{\FI}+(\ual\oD\om)_{\FI}&=(\de\ual\om)_{\FI}+(\ual\de\om)_{\FI}\\
&=\sum\limits_{k=0}^p\tf^{1,*}_{\FI}\big(\om_{\FI^{1,2}_{\hk,k}}\big|_{\tV_{\FI^1}}-\om_{\FI^{1,2}_{k,\hk}}\big|_{\tV_{\FI^1}}\big)\\
&=\tf^{1,*}_{\FI}(\om_{\FI^2}\big|_{\tV_{\FI^1}}-\om_{\FI^1})\\
&=\tf^{1,*}_{\FI}\al^*_{\FI}(\om_{\FI^2})-\tf^{1,*}_{\FI}(\om_{\FI^1})&(\text{Remark \ref{restrict-indep-embed}})\\
&=\tf^{2,*}_{\FI}(\om_{\FI^2})-\tf^{1,*}_{\FI}(\om_{\FI^1})&(\text{Eq.}\ \eqref{al-FI})\\
&=(\utf^2\om)_{\FI}-(\utf^1\om)_{\FI}.&(\text{Eq.}\ \eqref{utf-def})
\end{align*}

Second, it is clear that $\ual=0$ on $\oE^*_{\CV}\subseteq\oC^0(\CV,\Om^*)$ and $\ual$ maps $\oI^*_{\CV}$ into $\oI^*_{\CU}$. Hence, $\ual:\utf^1\to/=>/\utf^2$ is a spark homotopy.
\end{proof}

Moreover, $\oS$ is compatible with vertical and horizontal compositions of 2-cells:

\begin{prop} \label{nat-induce-homotopy-vcomp}
Let $\tf^1\to/=>/^{\al}\tf^2\to/=>/^{\be}\tf^3:\CU\to\CV$ be natural transformations of compatible systems. Then there exists a homotopy
$$\Ga:\ube+\ual\toRr\ubeal$$
of spark homotopies.
\end{prop}

\begin{proof}
Following the notation scheme in the proof of Proposition \ref{nat-induce-homotopy}, for $\FI=I_0\dots I_p\in\MpVU$ we further denote
$$\FI^{1,2,3}_{i,j}=I^1_0\dots I^1_i I^2_i\dots I^2_j I^3_j\dots I^3_p\in\MpppVU$$
and $\FI^{1,2,3}_{i,j,\hk}$, $\FI^{1,2,3}_{i,\hk,j}$, $\FI^{1,2,3}_{\hk,i,j}$ for their obvious meanings. Similar to \eqref{tV-FI12-embed} one has the embeddings
$$\tV_{\FI^1}\to\tV_{\FI^{1,2}_j}\to\tV_{\FI^2}\to\tV_{\FI^{2,3}_j}\to\tV_{\FI^3}\quad\text{and}\quad\tV_{\FI^1}\to\tV_{\FI^{1,2,3}_{i,j}}\to\tV_{\FI^3}$$
in $\CV$. For every $\om\in\oC^p(\CV,\Om^q)$, define
\begin{equation} \label{Ga-def}
\Ga\om\in\oC^{p-2}(\CU,\Om^q)\quad\text{with}\quad(\Ga\om)_{\FI}=\sum\limits_{0\leq i\leq j\leq p-2}(-1)^{i+j}\tf^{1,*}_{\FI}\big(\om_{\FI^{1,2,3}_{i,j}}\big|_{\tV_{\FI^1}}\big).
\end{equation}
We show that $\Ga$ defines a homotopy $\ube+\ual\toRr\ubeal$ of spark homotopies.

First, $\oD\circ\Ga-\Ga\circ\oD=\ubeal-(\ube+\ual)$. Since obviously $d\Ga=\Ga d$, by \eqref{oD-def} it suffices to prove
$$\de\circ\Ga-\Ga\circ\de=\ubeal-(\ube+\ual).$$
For any $\om\in\oC^p(\CV,\Om^q)$ and $\FI\in\MpmVU$, similar to the computations for $\de\ual$ and $\ual\de$ in the proof of Proposition \ref{nat-induce-homotopy} one has
\begin{align*}
(\de\Ga\om)_{\FI}={}&\sum\limits_{k=0}^{p-1}(-1)^k\tf^{1,*}_{\FI}\Big[\sum\limits_{0\leq i\leq j\leq p-2}(-1)^{i+j}\om_{(\FI_{\hk})^{1,2,3}_{i,j}}\big|_{\tV_{\FI^1}}\Big]\\
={}&\sum\limits_{0\leq i\leq j<k\leq p-1}(-1)^{i+j+k}\tf^{1,*}_{\FI}(\om_{\FI^{1,2,3}_{i,j,\hk}}\big|_{\tV_{\FI^1}})+\sum\limits_{0\leq i<k<j\leq p-1}(-1)^{i+j+k+1}\tf^{1,*}_{\FI}(\om_{\FI^{1,2,3}_{i,\hk,j}}\big|_{\tV_{\FI^1}})\\
{}&+\sum\limits_{0\leq k<i\leq j\leq p-1}(-1)^{i+j+k}\tf^{1,*}_{\FI}(\om_{\FI^{1,2,3}_{\hk,i,j}}\big|_{\tV_{\FI^1}})
\end{align*}
and
\begin{align*}
(\Ga\de\om)_{\FI}={}&\sum\limits_{0\leq i\leq j\leq p-1}(-1)^{i+j}\tf^{1,*}_{\FI}\Big[\sum\limits_{k=0}^{p+1}(-1)^k\om_{(\FI^{1,2,3}_{i,j})_{\hk}}\big|_{\tV_{\FI^1}}\Big]\\
={}&\sum\limits_{0\leq k\leq i\leq j\leq p-1}(-1)^{i+j+k}\tf^{1,*}_{\FI}(\om_{\FI^{1,2,3}_{\hk,i,j}}\big|_{\tV_{\FI^1}})+\sum\limits_{0\leq i\leq k\leq j\leq p-1}(-1)^{i+j+k+1}\tf^{1,*}_{\FI}(\om_{\FI^{1,2,3}_{i,\hk,j}}\big|_{\tV_{\FI^1}})\\
{}&+\sum\limits_{0\leq i\leq j\leq k\leq p-1}(-1)^{i+j+k}\tf^{1,*}_{\FI}(\om_{\FI^{1,2,3}_{i,j,\hk}}\big|_{\tV_{\FI^1}}).
\end{align*}
Note also that
$$\FI^{1,2,3}_{i,j,\hj}=\FI^{1,2,3}_{i,\widehat{j+1},j+1}\ \text{if}\ 0\leq i\leq j\leq p-2,\qquad\FI^{1,2,3}_{\hi,i,j}=\FI^{1,2,3}_{i-1,\widehat{i-1},j}\ \text{if}\ 1\leq i\leq j\leq p-1$$
and, by \eqref{al-FI},
$$\tf^{1,*}_{\FI}(\om_{\FI^{2,3}_j}\big|_{\tV_{\FI^1}})=\tf^{1,*}_{\FI}\al^*_{\FI}\lam^*_{\FI^2,\FI^{2,3}_j}(\om_{\FI^{2,3}_j})
=\tf^{2,*}_{\FI}(\om_{\FI^{2,3}_j}\big|_{\tV_{\FI^2}}).$$
Thus
\begin{align*}
{}&(\de\Ga\om)_{\FI}-(\Ga\de\om)_{\FI}\\
={}&\sum\limits_{0\leq i\leq j\leq p-1}(-1)^{i+1}\tf^{1,*}_{\FI}(\om_{\FI^{1,2,3}_{i,j,\hj}}\big|_{\tV_{\FI^1}})+\sum\limits_{0\leq i\leq j\leq p-1}(-1)^{j+1}\tf^{1,*}_{\FI}(\om_{\FI^{1,2,3}_{\hi,i,j}}\big|_{\tV_{\FI^1}})\\
{}&+\sum\limits_{0\leq i\leq j\leq p-1}(-1)^j\tf^{1,*}_{\FI}(\om_{\FI^{1,2,3}_{i,\hi,j}}\big|_{\tV_{\FI^1}})+\sum\limits_{0\leq i<j\leq p-1}(-1)^i\tf^{1,*}_{\FI}(\om_{\FI^{1,2,3}_{i,\hj,j}}\big|_{\tV_{\FI^1}})\\
={}&\sum\limits_{i=0}^{p-1}(-1)^{i+1}\tf^{1,*}_{\FI}(\om_{\FI^{1,2,3}_{i,p-1,\widehat{p-1}}}\big|_{\tV_{\FI^1}})
+\sum\limits_{j=0}^{p-1}(-1)^{j+1}\tf^{1,*}_{\FI}(\om_{\FI^{1,2,3}_{\hat{0},0,j}}\big|_{\tV_{\FI^1}})
+\sum\limits_{j=0}^{p-1}(-1)^j\tf^{1,*}_{\FI}(\om_{\FI^{1,2,3}_{j,\hj,j}}\big|_{\tV_{\FI^1}})\\
={}&\sum\limits_{i=0}^{p-1}(-1)^{i+1}\tf^{1,*}_{\FI}(\om_{\FI^{1,2}_i}\big|_{\tV_{\FI^1}})
+\sum\limits_{j=0}^{p-1}(-1)^{j+1}\tf^{1,*}_{\FI}(\om_{\FI^{2,3}_j}\big|_{\tV_{\FI^1}})
+\sum\limits_{j=0}^{p-1}(-1)^j\tf^{1,*}_{\FI}(\om_{\FI^{1,3}_j}\big|_{\tV_{\FI^1}})\\
={}&\sum\limits_{i=0}^{p-1}(-1)^{i+1}\tf^{1,*}_{\FI}(\om_{\FI^{1,2}_i}\big|_{\tV_{\FI^1}})
+\sum\limits_{j=0}^{p-1}(-1)^{j+1}\tf^{2,*}_{\FI}(\om_{\FI^{2,3}_j}\big|_{\tV_{\FI^2}})
+\sum\limits_{j=0}^{p-1}(-1)^j\tf^{1,*}_{\FI}(\om_{\FI^{1,3}_j}\big|_{\tV_{\FI^1}})\\
={}&(\ubeal\om)_{\FI}-((\ube\om)_{\FI}+(\ual\om)_{\FI}).
\end{align*}

Second, it is clear that $\Ga=0$ on $\oE^*_{\CV}\subseteq\oC^0(\CV,\Om^*)$ and $\Ga$ maps $\oI^*_{\CV}$ into $\oI^*_{\CU}$. Hence, $\Ga:\ube+\ual\toRr\ubeal$ is a homotopy of spark homotopies.
\end{proof}

\begin{prop} \label{nat-induce-homotopy-hcomp}
Let $\al:\tf^1\to/=>/\tf^2:\CU\to\CV$ and $\be:\tg^1\to/=>/\tg^2:\CV\to\CW$ be natural transformations of compatible systems. Then there exists a homotopy
$$\Ga:(\ual\circ\utg^1+\utf^2\circ\ube)\toRr\ubecal$$
of spark homotopies.
\end{prop}

\begin{proof}
Based on the notations introduced in the proofs of Propositions \ref{nat-induce-homotopy} and \ref{nat-induce-homotopy-vcomp}, we write
\begin{align*}
&\FI^{mn}:=\tg^n\FI^m=\tg^n\tf^m\FI,\\
&\FI^{mn,m'n'}_i:=I^{mn}_0\dots I^{mn}_i I^{m'n'}_i\dots I^{m'n'}_p,\\
&\FI^{mn,m'n',m''n''}_{i,j}:=I^{mn}_0\dots I^{mn}_i I^{m'n'}_i\dots I^{m'n'}_j I^{m''n''}_j\dots I^{m''n''}_p
\end{align*}
for $\FI=I_0\dots I_p\in\MpVU$ and $m,m',m'',n,n',n''=1,2$, and $\FI^{mn,m'n',m''n''}_{i,j,\hk}$, $\FI^{mn,m'n',m''n''}_{i,\hk,j}$, $\FI^{mn,m'n',m''n''}_{\hk,i,j}$ for their obvious meanings. For every $\om\in\oC^p(\CW,\Om^q)$, define
\begin{equation} \label{Xi-def}
\Xi\om\in\oC^{p-2}(\CU,\Om^q)\quad\text{with}\quad(\Xi\om)_{\FI}=\sum\limits_{0\leq i\leq j\leq p-2}(-1)^{i+j}(\tg^1\tf^1)^*_{\FI}\big(\om_{\FI^{11,21,22}_{i,j}}\big|_{\tW_{\FI^{11}}}\big).
\end{equation}
Similar to the proof of Proposition \ref{nat-induce-homotopy-vcomp}, the only non-trivial part for exhibiting $\Xi:(\ual\circ\utg^1+\utf^2\circ\ube)\toRr\ubecal$ as a homotopy of spark homotopies is to show that
$$\de\circ\Xi-\Xi\circ\de=\ubecal-(\ual\circ\utg^1+\utf^2\circ\ube).$$
Indeed, for any $\om\in\oC^p(\CW,\Om^q)$ and $\FI\in\MpmVU$, one may proceed as in Proposition \ref{nat-induce-homotopy-vcomp} to verify that
\begin{align*}
(\de\Xi\om)_{\FI}-(\Xi\de\om)_{\FI}={}&\sum\limits_{i=0}^{p-1}(-1)^{i+1}(\tg^1\tf^1)^*_{\FI}(\om_{\FI^{11,21,22}_{i,p-1,\widehat{p-1}}}\big|_{\tW_{\FI^{11}}})
+\sum\limits_{j=0}^{p-1}(-1)^{j+1}(\tg^1\tf^1)^*_{\FI}(\om_{\FI^{11,21,22}_{\hat{0},0,j}}\big|_{\tW_{\FI^{11}}})\\
{}&+\sum\limits_{j=0}^{p-1}(-1)^j(\tg^1\tf^1)^*_{\FI}(\om_{\FI^{11,21,22}_{j,\hj,j}}\big|_{\tW_{\FI^{11}}})\\
={}&\sum\limits_{i=0}^{p-1}(-1)^{i+1}(\tg^1\tf^1)^*_{\FI}(\om_{\FI^{11,21}_i}\big|_{\tW_{\FI^{11}}})
+\sum\limits_{j=0}^{p-1}(-1)^{j+1}(\tg^1\tf^1)^*_{\FI}(\om_{\FI^{21,22}_j}\big|_{\tW_{\FI^{11}}})\\
{}&+\sum\limits_{j=0}^{p-1}(-1)^j(\tg^1\tf^1)^*_{\FI}(\om_{\FI^{11,22}_j}\big|_{\tW_{\FI^{11}}}).
\end{align*}
Since
\begin{align*}
(\ubecal\om)_{\FI}&=\sum\limits_{j=0}^{p-1}(-1)^j(\tg^1\tf^1)^*_{\FI}(\om_{\FI^{11,22}_j}\big|_{\tW_{\FI^{11}}}),&(\text{Eq. \eqref{ual-def}})\\
(\ual\circ\utg^1\om)_{\FI}&=\sum\limits_{i=0}^{p-1}(-1)^i\tf^{1,*}_{\FI}((\utg^1\om)_{\FI^{1,2}_i}\big|_{\tV_{\FI^1}})&(\text{Eq. \eqref{ual-def}})\\
&=\sum\limits_{i=0}^{p-1}(-1)^i\tf^{1,*}_{\FI}\Big(\tg^{1,*}_{\FI^{1,2}_i}(\om_{\FI^{11,21}_i})\Big|_{\tV_{\FI^1}}\Big)&(\text{Eq. \eqref{utf-def}})\\
&=\sum\limits_{i=0}^{p-1}(-1)^i(\tg^1\tf^1)^*_{\FI}(\om_{\FI^{11,21}_i}\big|_{\tW_{\FI^{11}}})&(\text{Lemma \ref{lift-indep-embed}})
\end{align*}
and
\begin{align*}
(\utf^2\circ\ube\om)_{\FI})&=\tf^{2,*}_{\FI}((\ube\om)_{\FI^2})&(\text{Eq. \eqref{utf-def}})\\
&=\sum\limits_{j=0}^{p-1}(-1)^j\tf^{2,*}_{\FI}\,\tg^{1,*}_{\FI^2}\big(\om_{\FI^{21,22}_j}\big|_{\tW_{\FI^{21}}}\big)&(\text{Eq. \eqref{ual-def}})\\
&=\sum\limits_{j=0}^{p-1}(-1)^j\tf^{1,*}_{\FI}\,\al^*_{\FI}\,\tg^{1,*}_{\FI^2}\big(\om_{\FI^{21,22}_j}\big|_{\tW_{\FI^{21}}}\big)&(\text{Eq. \eqref{al-FI}})\\
&=\sum\limits_{j=0}^{p-1}(-1)^j\tf^{1,*}_{\FI}\,\tg^{1,*}_{\FI^1}\,(\tg^1\al_I)^*\big(\om_{\FI^{21,22}_j}\big|_{\tW_{\FI^{21}}}\big)&(\text{Eq. \eqref{compat-sys:lift-embed}})\\
&=\sum\limits_{j=0}^{p-1}(-1)^j(\tg^1\tf^1)^*_{\FI}(\tg^1\al_I)^*\big(\om_{\FI^{21,22}_j}\big|_{\tW_{\FI^{21}}}\big)&(\text{Eq. \eqref{composite-lift}})\\
&=\sum\limits_{j=0}^{p-1}(-1)^j(\tg^1\tf^1)^*_{\FI}(\om_{\FI^{21,22}_j}\big|_{\tW_{\FI^{11}}}),&(\text{Remark \ref{restrict-indep-embed}})
\end{align*}
it follows that $(\de\Xi\om)_{\FI}-(\Xi\de\om)_{\FI}=(\ubecal\om)_{\FI}-((\ual\circ\utg^1\om)_{\FI}+(\utf^2\circ\ube\om)_{\FI})$, which completes the proof.
\end{proof}

Therefore we have proved:

\begin{thm} \label{oS-2-functor}
$\oS:\GPreOrb^{\op}\to\SpCx$ is a 2-functor that sends
\begin{enumerate}[label={\rm(\arabic*)}]
\item \label{oS-2-functor:ob}
    each good atlas $\CU$ to the spark complex
    $$\oS_{\CU}=(\oF^*_{\CU},\oE^*_{\CU},\oI^*_{\CU}):=(\Tot(\oC^*(\CU,\Om^*)),\oOm^*(\CU),\oC^*(\CU,\bbZ)),$$
\item \label{oS-2-functor:1-cell}
    each compatible system $\tf:\CU\to\CV$ to the spark homomorphism
    $$\utf:\oS_{\CV}\to\oS_{\CU},$$
\item \label{oS-2-functor:2-cell}
    each natural transformation $\al:\tf^1\to/=>/\tf^2:\CU\to\CV$ of compatible systems to the spark homotopy class
    $$[\ual]:\utf^1\to/=>/\utf^2:\oS_{\CV}\to\oS_{\CU}.$$
\end{enumerate}
\end{thm}

\section{Spark characters on good atlases as graded rings}

Composing the 2-functor $\oS$ (see Theorem \ref{oS-2-functor}) with the 2-functor $\Char:\SpCx\to\GAb$ (see \eqref{Char}) gives rise to the \emph{spark character 2-functor}
$$\oH:\GPreOrb^{\op}\to\GAb$$
on good atlases as described below:

\begin{thm} \label{oH-2-functor}
$\oH:\GPreOrb^{\op}\to\GAb$ is a 2-functor that sends
\begin{enumerate}[label={\rm(\arabic*)}]
\item \label{oH-2-functor:ob}
    each good atlas $\CU$ to the graded abelian group
    $$\oH^*(\CU):=\hH(\oF^*_{\CU},\oE^*_{\CU},\oI^*_{\CU})$$
    of spark characters on $\CU$,
\item \label{oH-2-functor:1-cell}
    each compatible system $\tf:\CU\to\CV$ to the homomorphism
    $$\utf_*:\oH^*(\CV)\to\oH^*(\CU)$$
    of graded abelian groups,
\item \label{oH-2-functor:2-cell}
    each natural transformation $\al:\tf^1\to/=>/\tf^2:\CU\to\CV$ of compatible systems to the identity 2-cell on $\utf^{1,*}=\utf^{2,*}$.
\end{enumerate}
Moreover, the diagram
$$\bfig
\qtriangle<1000,400>[\GPreOrb^{\op}`\SpCx`\GAb;\oS`\oH`\Char]
\efig$$
is commutative.
\end{thm}

The set $\oH^*(\CU)$ of spark characters on a good atlas $\CU$ carries more structures than a graded abelian group; in fact, it is a graded commutative ring. To see this, first note that there is a \emph{cup product} $\cup$ defined on $\oF^*_{\CU}=\Tot(\oC^*(\CU,\Om^*))$ with
\begin{equation} \label{cup-def}
(\om\cup\eta)_{\FI}:=(-1)^{jn}\om_{\FI_{\leq m}}|_{\tU_{\FI}}\wedge\eta_{\FI_{\geq m}}|_{\tU_{\FI}}\in\oC^{m+n}(\CU,\Om^{j+k})
\end{equation}
for all $\om\in\oC^m(\CU,\Om^j)$, $\eta\in\oC^n(\CU,\Om^k)$, $\FI=I_0\dots I_{m+n}\in\MmnVU$, where
$$\FI_{\leq m}=I_0\dots I_m\quad\text{and}\quad\FI_{\geq m}=I_m\dots I_{m+n}:$$

\begin{prop} \label{cup-prop}
The cup product $\cup$ on $(\oF^*_{\CU},\oD)$ is well-defined, associative and satisfies the Leibniz rule, i.e.,
\begin{equation} \label{cup-Leibniz}
\oD(\om\cup\eta)=\oD\om\cup\eta+(-1)^{j+m}\om\cup\oD\eta
\end{equation}
for all $\om\in\oC^m(\CU,\Om^j)$, $\eta\in\oC^n(\CU,\Om^k)$. Moreover, $\cup$ induces the wedge product on $\oE^*_{\CU}$ and the multiplication of numbers on $\oI^*_{\CU}$.
\end{prop}


The ring structure on $\oH^*(\CU)$ is then induced by the cup product $\cup$ on $(\oF^*_{\CU},\oD)$:

\begin{prop} \label{HU-ring}
$\oH^*(\CU)$ is a graded commutative ring, in which the product is given by
$$[\om]\star[\eta]=(-1)^{(k+1)(l+1)}[\eta]\star[\om]=[\om\cup c+(-1)^{k+1}r\cup\eta]=[\om\cup s+(-1)^{k+1}e\cup\eta]\in\oH^{k+l+1}(\CU)$$
for all $[\om]\in\oH^k(\CU)$ with $\oD \om=e-r$ and $[\eta]\in\oH^l(\CU)$ with $\oD\eta=c-s$.
\end{prop}


The proofs of Propositions \ref{cup-prop} and \ref{HU-ring} are similar to \cite[Lemma 3.8]{Du2016} and \cite[Theorem 3.11]{Du2016}, respectively; so, the details are left to the readers.

The homomorphism $\utf_*:\oH^*(\CV)\to\oH^*(\CU)$ given in Theorem \ref{oH-2-functor}\ref{oH-2-functor:1-cell} is actually a homomorphism of graded commutative rings:

\begin{prop} \label{compat-sys-induce-ring-homo}
Each compatible system $\tf:\CU\to\CV$ induces a homomorphism
$$\utf_*:\oH^*(\CV)\to\oH^*(\CU)$$
of graded commutative rings.
\end{prop}

\begin{proof}
It suffices to show that the spark homomorphism $\utf:\oS_{\CV}\to\oS_{\CU}$ obtained in Proposition \ref{compat-sys-induce-spark} preserves the cup products in $\oF^*_{\CV}$ given by \eqref{cup-def}. Indeed, for any $\om\in\oC^m(\CV,\Om^j)$, $\eta\in\oC^n(\CV,\Om^k)$ and $\FI=I_0\dots I_{m+n}\in\MmnVU$,
\begin{align*}
(\utf\om\cup\utf\eta)_{\FI}&=(-1)^{jn}(\utf\om)_{\FI_{\leq m}}|_{\tU_{\FI}}\wedge(\utf\eta)_{\FI_{\geq m}}|_{\tU_{\FI}}&(\text{Eq. \eqref{cup-def}})\\
&=(-1)^{jn}\tf^*_{\FI_{\leq m}}(\om_{\tf\FI_{\leq m}})|_{\tU_{\FI}}\wedge\tf^*_{\FI_{\geq m}}(\eta_{\tf\FI_{\geq m}})|_{\tU_{\FI}}&(\text{Eq. \eqref{utf-def}})\\
&=(-1)^{jn}\tf^*_{\FI}(\om_{\tf\FI_{\leq m}}|_{\tV_{\tf\FI}})\wedge\tf^*_{\FI}(\eta_{\tf\FI_{\geq m}}|_{\tV_{\tf\FI}})&(\text{Lemma \ref{lift-indep-embed}})\\
&=\tf^*_{\FI}((-1)^{jn}\om_{\tf\FI_{\leq m}}|_{\tV_{\tf\FI}}\wedge\eta_{\tf\FI_{\geq m}}|_{\tV_{\tf\FI}})\\
&=\tf^*_{\FI}((\om\cup\eta)_{\tf\FI})&(\text{Eq. \eqref{cup-def}})\\
&=(\utf(\om\cup\eta))_{\FI},&(\text{Eq. \eqref{utf-def}})
\end{align*}
as desired.
\end{proof}

Let $\GCRng$ denote the category of $(\bbZ\text{-})$graded commutative rings and their homomorphisms, also considered as a 2-category only equipped with trivial 2-cells. Then, with Proposition \ref{compat-sys-induce-ring-homo} one may enhance Theorem \ref{oH-2-functor} to the following:

\begin{thm} \label{oH-2-functor-ring}
$\oH:\GPreOrb^{\op}\to\GCRng$ is a 2-functor that sends
\begin{enumerate}[label={\rm(\arabic*)}]
\item \label{oH-2-functor-ring:ob}
    each good atlas $\CU$ to the graded commutative ring
    $$\oH^*(\CU):=\hH(\oF^*_{\CU},\oE^*_{\CU},\oI^*_{\CU})$$
    of spark characters on $\CU$,
\item \label{oH-2-functor-ring:1-cell}
    each compatible system $\tf:\CU\to\CV$ to the homomorphism
    $$\utf_*:\oH^*(\CV)\to\oH^*(\CU)$$
    of graded commutative rings,
\item \label{oH-2-functor-ring:2-cell}
    each natural transformation $\al:\tf^1\to/=>/\tf^2:\CU\to\CV$ of compatible systems to the identity 2-cell on $\utf^{1,*}=\utf^{2,*}$.
\end{enumerate}
\end{thm}

\appendix

\section{Appendix: a different spark complex on a good atlas} \label{Du-Zhao}

Given a good atlas $\CU=\{(\tU_I,G_I,\pi_I)\mid I\in\FVU\}$, a different spark complex can be constructed on $\CU$ as considered in \cite{Du2016}. Explicitly, let
$$\mVU$$
denote the free monoid on the set $V_{\CU}$, whose elements are strings
$$I=i_0\dots i_p$$
consisting of elements of $V_{\CU}$. Then, since $\mVU$ can be embedded into $\MVU$ in the obvious way as a submonoid, using the same notation scheme as in Remark \ref{FI-notation}\ref{FI-notation:k}--\ref{FI-notation:MpVU} one may define
$$\sC^p(\CU,\Om^q):=\Big\{(\om_I)\in\prod_{I\in\mpVU}\Om^q(\tU_I)^{G_I}\mathrel{\Big|}\om_I=-\om_{I_{j\mathrel{\leftrightarrow}k}}\ \text{whenever}\ 0\leq j,k\leq p\Big\},$$
which gives rise to a spark complex \cite{Du2016}
\begin{equation} \label{SU-def}
\CS_{\CU}=(\sF^*_{\CU},\sE^*_{\CU},\sI^*_{\CU}):=(\Tot(\sC^*(\CU,\Om^*)),\Om^*(\CU),\sC^*(\CU,\bbZ))
\end{equation}
in the same way as the construction of $\oS_{\CU}$ in Theorem \ref{oSU-def}. In particular, $\Om^q(\CU)$ consists of families
$$\{\om_i\in\Om^q(\tU_i)\mid i\in V_{\CU}\}$$
of differential $q$-forms such that
\begin{enumerate}[label=(\arabic*)]
\item each $\om_i$ is $G_i$-invariant, and
\item if $U_i\cap U_j\neq\varnothing$ and there exist embeddings
    $$(\tU_i,G_i,\pi_i)\toleft^{\lam_i}(\tU_{ij},G_{ij},\pi_{ij})\to^{\lam_j}(\tU_j,G_j,\pi_j),$$
    then $\lam^*_i(\om_i)=\lam^*_j(\om_j)$.
\end{enumerate}

Recall that a spark homomorphism $f:(\sF^*,\sE^*,\sI^*)\to(\oF^*,\oE^*,\oI^*)$ of spark complexes is a \emph{quasi-isomorphism} \cite{Hao2009} if
\begin{enumerate}[label=(\arabic*)]
\item $f$ is injective,
\item $f|_{\sI^*}$ induces an isomorphism $H^*(\sI^*)\cong H^*(\oI^*)$, and
\item $f|_{\sE^*}$ is an isomorphism\footnote{Our definition of quasi-isomorphisms here deviates a little bit from \cite{Hao2009}, where $f|_{\sE^*}$ is required to be an identity map. The prototype of this notion comes from \emph{subspark complexes} defined in \cite{Harvey2006}.}.
\end{enumerate}

In what follows we show that $\CS_{\CU}$ and $\oS_{\CU}$ are quasi-isomorphic spark complexes. Indeed, every map
$$\phi:\FVU\to V_{\CU}\quad\text{with}\quad\phi I\in I$$
gives rise to a quasi-isomorphism $\ophi:\CS_{\CU}\to\oS_{\CU}$ by extending each $\om\in\sC^p(\CU,\Om^q)$ to $\ophi\om\in\oC^p(\CU,\Om^q)$ with
\begin{equation} \label{phi-star-def}
(\ophi\om)_{\FI}=\om_{\phi\FI}|_{\tU_{\FI}}
\end{equation}
for all $\FI=I_0\dots I_p\in\MpVU$, where $\phi\FI:=(\phi I_0)\dots(\phi I_p)$ and, obviously, $\ophi\om$ is independent of the choice of the embedding $\lam_{\FI,\phi\FI}:(\tU_{\FI},G_{\FI},\pi_{\FI})\to(\tU_{\phi\FI},G_{\phi\FI},\pi_{\phi\FI})$ by Remark \ref{restrict-indep-embed}:

\begin{prop} \label{phi-quasi-iso}
$\ophi:\CS_{\CU}\to\oS_{\CU}$ is a quasi-isomorphism of spark complexes.
\end{prop}

\begin{proof}
First, $\ophi:\sF^*_{\CU}\to\oF^*_{\CU}$ is an injective cochain map. Let $D$ denote the differential on $\Tot(\sC^*(\CU,\Om^*))$. Then for all $\om\in\sC^p(\CU,\Om^q)$ and $\FI\in\MppVU$,
\begin{align*}
(\oD\circ\ophi\om)_{\FI}&=\sum\limits_{k=0}^{p+1}(-1)^k(\ophi\om)_{\FI_{\hk}}\big|_{\tU_{\FI}}+(-1)^p(d\ophi\om)_{\FI}&(\text{Eq. \eqref{delta-def} \& \eqref{oD-def}})\\
&=\sum\limits_{k=0}^{p+1}(-1)^k\om_{\phi\FI_{\hk}}\big|_{\tU_{\FI_{\hk}}}\big|_{\tU_{\FI}}+(-1)^p(d\om_{\phi\FI}|_{\tU_{\FI}})&(\text{Eq. \eqref{phi-star-def}})\\
&=\sum\limits_{k=0}^{p+1}(-1)^k\om_{(\phi\FI)_{\hk}}\big|_{\tU_{\FI}}+(-1)^p(d\om_{\phi\FI}|_{\tU_{\FI}})&(\text{Remark \ref{restrict-indep-embed}})\\
&=\sum\limits_{k=0}^{p+1}(-1)^k\om_{(\phi\FI)_{\hk}}\big|_{\tU_{\phi\FI}}\big|_{\tU_{\FI}}+(-1)^p(d\om)_{\phi\FI}\big|_{\tU_{\FI}}&(\text{Remark \ref{restrict-indep-embed}})\\
&=(\ophi\circ D\om)_{\FI},&(\text{Eq. \eqref{delta-def}, \eqref{oD-def}  \& \eqref{phi-star-def}})
\end{align*}
and the injectivity of $\ophi$ is obvious.

Second, $\ophi|_{\sI^*}$ induces an isomorphism $H^*(\sI^*)\cong H^*(\oI^*)$. Since $\{U_i\mid i\in V_{\CU}\}$ and $\{U_I\mid I\in\FVU\}$ are both good covers of $X$, it follows from Leray's theorem for sheaf cohomology (see \cite[Theorem III.4.13]{Bredon1997}) that
\begin{align*}
H^*(\sC^*(\CU,\bbZ))&\cong\check{H}^*(\{U_i\mid i\in V_{\CU}\};\bbZ)\cong H^*(X;\bbZ)\\
&\cong\check{H}^*(\{U_I\mid I\in\FVU\};\bbZ)\cong H^*(\oC^*(\CU,\bbZ)).
\end{align*}
Hence, $\ophi|_{\sI^*}:H^*(\sC^*(\CU,\bbZ))\cong H^*(\oC^*(\CU,\bbZ))$ follows from the general construction of the {\v C}ech cohomology (see \cite[Chapter 5]{Warner1983}).

Finally, $\ophi|_{\sE^*}$ is an isomorphism. Note that every $\{\om_i\in\Om^q(\tU_i)\mid i\in V_{\CU}\}\in\Om^q(\CU)$ uniquely extends to a $q$-form $\ophi\om\in\oOm^q(\CU)$ with
$$(\ophi\om)_I=\om_{\phi I}|_{\tU_I}$$
for all $I\in\FVU$, which is independent of choices of the map $\phi$ by the definition of $\Om^q(\CU)$. Hence, $\ophi|_{\sE^*}:\Om^*(\CU)\to\oOm^*(\CU)$ is an isomorphism.
\end{proof}

Proposition \ref{phi-quasi-iso} in conjunction with the following lemma, which can be proved similarly to \cite[Proposition 2.2.3]{Hao2009}, shows that the graded abelian group
$$\hH^*(\CU):=\hH(\sF^*_{\CU},\sE^*_{\CU},\sI^*_{\CU})$$
of spark characters of the spark complex $\CS_{\CU}$ is isomorphic to $\oH^*(\CU)$ (cf. Proposition \ref{Char-functor}):

\begin{lem} {\rm\cite{Hao2009}}
A quasi-isomorphism $f:(\sF^*,\sE^*,\sI^*)\to(\oF^*,\oE^*,\oI^*)$ of spark complexes induces an isomorphism
$$f_*:\hH(\sF^*,\sE^*,\sI^*)\cong\hH(\oF^*,\oE^*,\oI^*),\quad [a]\mapsto[fa]$$
of the associated graded abelian groups of spark characters.
\end{lem}

Furthermore, similar to Proposition \ref{HU-ring} one may define a ring structure on $\hH^*(\CU)$, and the isomorphism
$$\hH^*(\CU)\cong\oH^*(\CU)$$
is in fact a ring isomorphism:

\begin{prop}
$\ophi_*:\hH^*(\CU)\to\oH^*(\CU)$ is an isomorphism of graded commutative rings.
\end{prop}

\begin{proof}
It suffices to prove that $\ophi$ preserves the cup products on $\sF^*_{\CU}$, also defined by \eqref{cup-def}. Indeed, for any $\om\in\sC^m(\CV,\Om^j)$, $\eta\in\sC^n(\CV,\Om^k)$ and $\FI=I_0\dots I_{m+n}\in\MmnVU$,
\begin{align*}
(\ophi\om\cup\ophi\eta)_{\FI}&=(-1)^{jn}(\ophi\om)_{\FI_{\leq m}}|_{\tU_{\FI}}\wedge(\ophi\eta)_{\FI_{\geq m}}|_{\tU_{\FI}}&(\text{Eq. \eqref{cup-def}})\\
&=(-1)^{jn}\om_{\phi\FI_{\leq m}}\big|_{\tU_{\FI_{\leq m}}}\big|_{\tU_{\FI}}\wedge\eta_{\phi\FI_{\geq m}}\big|_{\tU_{\FI_{\geq m}}}\big|_{\tU_{\FI}}&(\text{Eq. \eqref{phi-star-def}})\\
&=(-1)^{jn}\om_{(\phi\FI)_{\leq m}}\big|_{\tU_{\FI}}\wedge\eta_{(\phi\FI)_{\geq m}}\big|_{\tU_{\FI}}&(\text{Remark \ref{restrict-indep-embed}})\\
&=(-1)^{jn}\om_{(\phi\FI)_{\leq m}}\big|_{\tU_{\phi\FI}}\big|_{\tU_{\FI}}\wedge\eta_{(\phi\FI)_{\geq m}}\big|_{\tU_{\phi\FI}}\big|_{\tU_{\FI}}&(\text{Remark \ref{restrict-indep-embed}})\\
&=(\ophi(\om\cup\eta))_{\FI},&(\text{Eq. \eqref{cup-def} \& \eqref{phi-star-def}})
\end{align*}
which completes the proof.
\end{proof}


\end{document}